\documentclass[12pt]{article}
\newcommand{\authorfootnotes}{\renewcommand\thefootnote{\@fnsymbol\c@footnote}}%
\usepackage{amssymb,amsmath, amsfonts, amsthm}

\usepackage{graphicx}
\usepackage{tikz}
\usepackage{enumerate}

\newcommand{\cp}{\,\square\,}

\newcommand{\Z}{\mathbb{Z}}
\newcommand{\Q}{\mathbb{Q}}
\newcommand{\R}{\mathbb{R}}
\newcommand{\C}{\mathbb{C}}
\newcommand{\N}{\mathbb{N}}
\newcommand{\chg}[2]{{\cal Q}_#1^{#2}}
\newcommand{\eline}{\begin{align*}\hline\end{align*}}

\newcommand{\vertex}{\node[vertex]}
\tikzstyle{vertex}=[circle, draw, inner sep=0pt, minimum size=6pt]

\usetikzlibrary{arrows}

\newtheorem{thm}{Theorem}[section]
\newtheorem{lem}[thm]{Lemma}
\newtheorem{cor}[thm]{Corollary}
\newtheorem{prop}[thm]{Proposition}

\newcommand{\ideg}{{\rm indeg}}
\newcommand{\odeg}{{\rm outdeg}}

\newcommand{\D}{{\rm DOM}}

\newcommand{\ugr}[1]{\widehat{#1}}
\newcommand{\DR}[1]{{\textcolor{blue}{#1}}}

\def\keywords{\vspace{.5em}
{\textit{Keywords}:\,\relax%
}}
\def\endkeywords{\par}
\def\class{\vspace{.5em}
{\textit{AMS classification}:\,\relax%
}}
\def\endclass{\par}

\begin{document}

\title{The Cut Method on Hypergraphs for the Wiener Index}

\author{
Sandi Klav\v zar $^{a,b,c}$  \and 	Ga\v{s}per Domen Romih $^{a}$
}

\date{\today}

\maketitle

\begin{center}
	$^a$ Faculty of Mathematics and Physics, University of Ljubljana, Slovenia\\
	{\tt sandi.klavzar@fmf.uni-lj.si}\\
	{\tt gasperdomen.romih@fmf.uni-lj.si}\\
	\medskip
	
	$^b$ Faculty of Natural Sciences and Mathematics, University of Maribor, Slovenia\\
	\medskip
	
	$^c$ Institute of Mathematics, Physics and Mechanics, Ljubljana, Slovenia\\
	
\end{center}

\begin{abstract}
The cut method has been proved to be extremely useful in chemical graph theory. In this paper the cut method is extended to hypergraphs. More precisely, the method is developed for the Wiener index of $k$-uniform partial cube-hypergraphs. The method is applied to cube-hypergraphs and hypertrees. Extensions of the method to hypergraphs arising in chemistry which are not necessary $k$-uniform and/or not necessary linear are also developed.
\end{abstract}


\maketitle

\section{Introduction}

The cut method, whose standard form was introduced in 1995 in~\cite{klavzar-1995}, has had a remarkable response in chemical graph theory. The method originally designed for the Wiener index of partial cubes was later developed for many other topological indices and has undergone many generalizations to more general situations than partial cubes. This applies in itself to many applications in mathematical chemistry where topological indices play important role. The basic idea is to first find a partition of the edges of a (molecular) graph and by removing parts of this partition construct smaller (weighted) graphs, called quotient graphs. After that, we infer back to the original graph from the quotient graphs.  The state of research on the cut method up to 2015 is summarized in the survey article~\cite{klavzar-2015}.  The method is still the subject of ongoing research, see~\cite{akhter-2021, arockiaraj-2022, arockiaraj-2016, brezovnik-2021, crepnjak-2017, tratnik-2020, tratnik-2021} as well as references therein. 
  
Hypergraphs form a structure that greatly generalizes the concept of a graph. In chemical graph theory, the standard method of representing molecules is by means of associated (chemical) graphs.  However, some molecules are more complicated than others and sometimes it is more convenient and more adequate to represent them as hypergraphs, see~\cite{gutman-1999, konstantinova-2001} for some chemical problems dealing with hypergraph theory. As a result, various problems of importance in mathematical chemistry have been investigated on hypergraphs, including spectral aspects~\cite{andreoti-2022, lin-2018, saha-2022} and different topological indices~\cite{weng-2020, weng-2022}. Very recently, while investigating molecular representations in drug design, a hypergraph-based topological framework was designed to characterize molecular structures and interactions at atomic level~\cite{liu-2021}. Interestingly, in the very same year when the cut method was introduced, Burosch and Ceccherini published the paper~\cite{burosch-1995} on isometric embeddings into hypergraphs, which is the second main source for the present paper.

The Wiener index is one of the most researched topics in the whole field of chemical graph theory. As already mentioned, the cut method was first designed for the Wiener index of graphs. In the last few years, the Wiener index has received a lot of attention also on hypergraphs. In~\cite{sun-2017} the authors investigate, among others, $3$-uniform paths and lower bounds on the Wiener index of $k$-uniform hypergraphs. In~\cite{dobrynin-2019, li-2020b} hypergraphs are constructed from trees and their Wiener index investigated. The effect of some transformations on the Wiener index of a hypergraph and extremal hypertrees with respect to the Wiener index is studied in~\cite{liu-2020}. The $k$-uniform unicyclic hypergraphs with maximum/minimum and second maximum/minimum Wiener index are determined in~\cite{zou-2020}, while the Wiener index of some composite hypergraphs and sunflower hypergraphs is the topic of~\cite{ashraf-2022}. Finally, in~\cite{che-2022} the concept of the $k$-Wiener index is introduced and studied on the so called $k$-plex hypergraphs.

 We proceed as follows. In the next section we introduce the mathematical machinery on hypergraphs needed latter on. In particular, partial cube-hypergraphs are defined and their characterizations recalled. In Section~\ref{sec:cut} we develop the cut method for the Wiener index of a hypergraph. In the last section we provide applications and extensions of the cut method including cube-hypergraphs, hypertrees, and the so called linear phenylene hypergraphs.
 
\section{Preliminaries}
In this section, we set the scene for the hypergraph cut method. In the first part, we introduce the necessary concepts about hypergraphs, focusing on distance and their Cartesian products. We then introduce partial cube-hypergraphs on which the cut method will operate and recall two of their characterizations. 
\subsection{Hypergraphs}

A hypergraph $H=(V(H),E(H))$ has the vertex set $V(H)$ and the edge set $E(H)$, where each edge $e \in E(H)$ is a non-empty subset of $V(H)$. $H$ is {\em $k$-uniform} if the size of every edge $e \in E(H)$ is $k$ and is {\em linear} if $|e \cap e'| \le 1$ for every $e,e' \in E(H)$, $e \neq e'$. Let $H$ and $H'$ be hypergraphs. If $V(H') \subseteq V(H)$ and $E(H') \subseteq E(H)$ we say that $H' \subseteq H$ is a {\em subhypergraph} of $H$. Clearly, if $H$ is $k$-uniform, then $H'$ is also $k$-uniform. If $F \subseteq E(H)$, then $H - F$ denotes the subhypergraph of $H$ obtained from $H$ by removing all the edges from $F$.

Let $u$ and $v$ be different vertices of $H$. A {\em $u,v$-path of length} $s \ge 1$  in $H$ is a sequence $u_0 = u, e_1, u_1, \ldots, e_s, u_{s} = v$, where $u_i$ are pairwise different vertices, $e_i$ are pairwise different edges, and $\{u_{i-1}, u_i \} \subseteq e_i $ for $i \in [s] = \{1,\ldots, s\}$. The {\em distance} $d_H(u, v)$ between vertices $u$ and $v$ is the length of a shortest $u,v$-path. We also set $d_H(u,u) = 0$. A subhypergraph $H' \subseteq H$ is {\em isometric} if $d_{H'}(u,v) = d_H(u,v)$ holds for all $u,v \in V(H')$. We further say that a set of vertices $X \subseteq V(H)$ is {\em convex} in $H$ if for every $u,v \in X$ and every $z \in V(H)$, the equality $d_H(x,z) + d_H(z, y) = d_H(x,y)$ implies $z \in X$. The {\em Wiener index} of a hypergraph $H$ is defined as the sum of the distances between all unordered pairs of vertices of $H$, that is,
$$
W(H) = \sum_{\{u,v\} \in \binom{V(H) }{2}}d_H(u,v).
$$

The {\em Cartesian product} $H \cp H'$ of hypergraphs $H$ and $H'$ is a hypergraph  with the vertex set $V(H) \times V(H')$ and the edge set 
$$\{\{u\} \times e' :\ u \in V(H),\ e' \in E(H') \} \cup \{ e \times \{u'\}  :\ e \in E(H),\ u' \in V(H')  \}.$$
 Just as Cartesian products of graphs, Cartesian products of hypergraphs have several nice properties, cf ~\cite{hammack-2016, hellmuth-2016}. In particular, if $H$ and $H'$ are $k$-uniform  hypergraphs, then $H \cp H'$ is also $k$-uniform, and the Cartesian product operation is associative. For $k \ge 2$, let ${\cal Q}_k$ denote the hypergraph with $k$ vertices and a single edge containing all the vertices. For $n \ge 1$, the {\em $k$-uniform $n$-cube} ${\cal Q}_k^{n}$ is the Cartesian product of $k$ copies of ${\cal Q}_k$. See Fig.~\ref{fig:cubes} where $\chg{3}{1}$, $\chg{3}{2}$, and $\chg{3}{3}$ are presented.

\begin{figure}[ht!]
\centering

\begin{tikzpicture}[scale=0.45]
	\begin{scope}[name=hedge]
		\def\vr{2pt}
		\node  (0) at (0, 0.5) {};
		\node  (1) at (0, 0) {};
		\node  (2) at (7, 0) {};
		\node (3) at (7, 0.5) {};
		\draw (0.center) to (3.center);
		\draw [in=180, out=0] (1.center) to (2.center);
		\draw [bend left=90, looseness=1.50] (3.center) to (2.center);
		\draw [bend right=90, looseness=1.50] (0.center) to (1.center);		
		\draw (0,0.25)  [fill=black] circle (\vr);
		\draw (3.5,0.25)  [fill=black] circle (\vr);
		\draw (3.5,-1)  [fill=none] circle node {$\chg{3}{1}$};
		\draw (7,0.25)  [fill=black] circle (\vr);
	\end{scope}
\end{tikzpicture}
\hspace{3pt}
\begin{tikzpicture}[scale=0.45]
	\begin{scope}[rotate=90,xshift=0.211cm, yshift=-0.27cm]
		\def\vr{2pt}
		\node  (0) at (0, 0.5) {};
		\node  (1) at (0, 0) {};
		\node  (2) at (7, 0) {};
		\node (3) at (7, 0.5) {};
		\draw (0.center) to (3.center);
		\draw [in=180, out=0] (1.center) to (2.center);
		\draw [bend left=90, looseness=1.50] (3.center) to (2.center);
		\draw [bend right=90, looseness=1.50] (0.center) to (1.center);		
	\end{scope}
	\begin{scope}[rotate=90,xshift=0.211cm, yshift=-3.72cm]
		\def\vr{2pt}
		\node  (0) at (0, 0.5) {};
		\node  (1) at (0, 0) {};
		\node  (2) at (7, 0) {};
		\node (3) at (7, 0.5) {};
		\draw (0.center) to (3.center);
		\draw [in=180, out=0] (1.center) to (2.center);
		\draw [bend left=90, looseness=1.50] (3.center) to (2.center);
		\draw [bend right=90, looseness=1.50] (0.center) to (1.center);		
	\end{scope}

\begin{scope}[rotate=90,xshift=0.211cm, yshift=-7.2cm]
		\def\vr{2pt}
		\node  (0) at (0, 0.5) {};
		\node  (1) at (0, 0) {};
		\node  (2) at (7, 0) {};
		\node (3) at (7, 0.5) {};
		\draw (0.center) to (3.center);
		\draw [in=180, out=0] (1.center) to (2.center);
		\draw [bend left=90, looseness=1.50] (3.center) to (2.center);
		\draw [bend right=90, looseness=1.50] (0.center) to (1.center);		
	\end{scope}
	\begin{scope}[name=hedge,yshift=3.5cm]
		\def\vr{2pt}
		\node  (0) at (0, 0.5) {};
		\node  (1) at (0, 0) {};
		\node  (2) at (7, 0) {};
		\node (3) at (7, 0.5) {};
		\draw (0.center) to (3.center);
		\draw [in=180, out=0] (1.center) to (2.center);
		\draw [bend left=90, looseness=1.50] (3.center) to (2.center);
		\draw [bend right=90, looseness=1.50] (0.center) to (1.center);		
		\draw (0,0.25)  [fill=black] circle (\vr);
		\draw (3.5,0.25)  [fill=black] circle (\vr);
		
		\draw (7,0.25)  [fill=black] circle (\vr);
	\end{scope}
		\begin{scope}[name=hedge,yshift=6.9cm]
		\def\vr{2pt}
		\node  (0) at (0, 0.5) {};
		\node  (1) at (0, 0) {};
		\node  (2) at (7, 0) {};
		\node (3) at (7, 0.5) {};
		\draw (0.center) to (3.center);
		\draw [in=180, out=0] (1.center) to (2.center);
		\draw [bend left=90, looseness=1.50] (3.center) to (2.center);
		\draw [bend right=90, looseness=1.50] (0.center) to (1.center);		
		\draw (0,0.25)  [fill=black] circle (\vr);
		\draw (3.5,0.25)  [fill=black] circle (\vr);
		\draw (7,0.25)  [fill=black] circle (\vr);
	\end{scope}
	
	\begin{scope}[name=hedge]
		\def\vr{2pt}
		\node  (0) at (0, 0.5) {};
		\node  (1) at (0, 0) {};
		\node  (2) at (7, 0) {};
		\node (3) at (7, 0.5) {};
		\draw (0.center) to (3.center);
		\draw [in=180, out=0] (1.center) to (2.center);
		\draw [bend left=90, looseness=1.50] (3.center) to (2.center);
		\draw [bend right=90, looseness=1.50] (0.center) to (1.center);		
		\draw (0,0.25)  [fill=black] circle (\vr);
		\draw (3.5,0.25)  [fill=black] circle (\vr);
		\draw (3.5,-1)  [fill=none]  node {$\chg{3}{2}$};
		\draw (7,0.25)  [fill=black] circle (\vr);
	\end{scope}
\end{tikzpicture}
\hspace{2pt}
\begin{tikzpicture}[scale=0.45]

	\begin{scope}

	\begin{scope}[rotate=60,xshift=0.211cm, yshift=-0.13cm]
		\def\vr{2pt}
		\node  (0) at (0, 0.5) {};
		\node  (1) at (0, 0) {};
		\node  (2) at (4, 0) {};
		\node (3) at (4, 0.5) {};
		\draw (0.center) to (3.center);
		\draw [in=180, out=0] (1.center) to (2.center);
		\draw [bend left=90, looseness=1.50] (3.center) to (2.center);
		\draw [bend right=90, looseness=1.50] (0.center) to (1.center);		
	\end{scope}
	
		\begin{scope}[rotate=60,xshift=1.9cm, yshift=-3.13cm]
		\def\vr{2pt}
		\node  (0) at (0, 0.5) {};
		\node  (1) at (0, 0) {};
		\node  (2) at (4, 0) {};
		\node (3) at (4, 0.5) {};
		\draw (0.center) to (3.center);
		\draw [in=180, out=0] (1.center) to (2.center);
		\draw [bend left=90, looseness=1.50] (3.center) to (2.center);
		\draw [bend right=90, looseness=1.50] (0.center) to (1.center);		
	\end{scope}
	
		\begin{scope}[rotate=60,xshift=3.63cm, yshift=-6.13cm]
		\def\vr{2pt}
		\node  (0) at (0, 0.5) {};
		\node  (1) at (0, 0) {};
		\node  (2) at (4, 0) {};
		\node (3) at (4, 0.5) {};
		\draw (0.center) to (3.center);
		\draw [in=180, out=0] (1.center) to (2.center);
		\draw [bend left=90, looseness=1.50] (3.center) to (2.center);
		\draw [bend right=90, looseness=1.50] (0.center) to (1.center);		
	\end{scope}

	\begin{scope}[name=hedge]
		\def\vr{2pt}
		\node  (0) at (0, 0.5) {};
		\node  (1) at (0, 0) {};
		\node  (2) at (7, 0) {};
		\node (3) at (7, 0.5) {};
		\draw (0.center) to (3.center);
		\draw [in=180, out=0] (1.center) to (2.center);
		\draw [bend left=90, looseness=1.50] (3.center) to (2.center);
		\draw [bend right=90, looseness=1.50] (0.center) to (1.center);		
		\draw (0,0.25)  [fill=black] circle (\vr) ;
		\draw (3.5,0.25)  [fill=black] circle (\vr);
		\draw (3.5,-1)  [fill=none] node {$\chg{3}{3}$};
		\draw (7,0.25)  [fill=black] circle (\vr);
	\end{scope}
	
	\begin{scope}[name=hedge,xshift=0.95cm, yshift=1.72cm]
		\def\vr{2pt}
		\node  (0) at (0, 0.5) {};
		\node  (1) at (0, 0) {};
		\node  (2) at (7, 0) {};
		\node (3) at (7, 0.5) {};
		\draw (0.center) to (3.center);
		\draw [in=180, out=0] (1.center) to (2.center);
		\draw [bend left=90, looseness=1.50] (3.center) to (2.center);
		\draw [bend right=90, looseness=1.50] (0.center) to (1.center);		
		\draw (0,0.25)  [fill=black] circle (\vr);
		\draw (3.5,0.25)  [fill=black] circle (\vr);
		\draw (7,0.25)  [fill=black] circle (\vr);
	\end{scope}
\begin{scope}[name=hedge,xshift=1.9cm, yshift=3.4cm]
		\def\vr{2pt}
		\node  (0) at (0, 0.5) {};
		\node  (1) at (0, 0) {};
		\node  (2) at (7, 0) {};
		\node (3) at (7, 0.5) {};
		\draw (0.center) to (3.center);
		\draw [in=180, out=0] (1.center) to (2.center);
		\draw [bend left=90, looseness=1.50] (3.center) to (2.center);
		\draw [bend right=90, looseness=1.50] (0.center) to (1.center);		
		\draw (0,0.25)  [fill=black] circle (\vr);
		\draw (3.5,0.25)  [fill=black] circle (\vr);
		\draw (7,0.25)  [fill=black] circle (\vr);
	\end{scope}
\end{scope}

	\begin{scope}[yshift = 4.5cm]

	\begin{scope}[rotate=60,xshift=0.211cm, yshift=-0.13cm]
		\def\vr{2pt}
		\node  (0) at (0, 0.5) {};
		\node  (1) at (0, 0) {};
		\node  (2) at (4, 0) {};
		\node (3) at (4, 0.5) {};
		\draw (0.center) to (3.center);
		\draw [in=180, out=0] (1.center) to (2.center);
		\draw [bend left=90, looseness=1.50] (3.center) to (2.center);
		\draw [bend right=90, looseness=1.50] (0.center) to (1.center);		
	\end{scope}
	
		\begin{scope}[rotate=60,xshift=1.9cm, yshift=-3.13cm]
		\def\vr{2pt}
		\node  (0) at (0, 0.5) {};
		\node  (1) at (0, 0) {};
		\node  (2) at (4, 0) {};
		\node (3) at (4, 0.5) {};
		\draw (0.center) to (3.center);
		\draw [in=180, out=0] (1.center) to (2.center);
		\draw [bend left=90, looseness=1.50] (3.center) to (2.center);
		\draw [bend right=90, looseness=1.50] (0.center) to (1.center);		
	\end{scope}
	
		\begin{scope}[rotate=60,xshift=3.63cm, yshift=-6.13cm]
		\def\vr{2pt}
		\node  (0) at (0, 0.5) {};
		\node  (1) at (0, 0) {};
		\node  (2) at (4, 0) {};
		\node (3) at (4, 0.5) {};
		\draw (0.center) to (3.center);
		\draw [in=180, out=0] (1.center) to (2.center);
		\draw [bend left=90, looseness=1.50] (3.center) to (2.center);
		\draw [bend right=90, looseness=1.50] (0.center) to (1.center);		
	\end{scope}

	\begin{scope}[name=hedge]
		\def\vr{2pt}
		\node  (0) at (0, 0.5) {};
		\node  (1) at (0, 0) {};
		\node  (2) at (7, 0) {};
		\node (3) at (7, 0.5) {};
		\draw (0.center) to (3.center);
		\draw [in=180, out=0] (1.center) to (2.center);
		\draw [bend left=90, looseness=1.50] (3.center) to (2.center);
		\draw [bend right=90, looseness=1.50] (0.center) to (1.center);		
		\draw (0,0.25)  [fill=black] circle (\vr);
		\draw (3.5,0.25)  [fill=black] circle (\vr);
		\draw (7,0.25)  [fill=black] circle (\vr);
	\end{scope}
	
	\begin{scope}[name=hedge,xshift=0.95cm, yshift=1.72cm]
		\def\vr{2pt}
		\node  (0) at (0, 0.5) {};
		\node  (1) at (0, 0) {};
		\node  (2) at (7, 0) {};
		\node (3) at (7, 0.5) {};
		\draw (0.center) to (3.center);
		\draw [in=180, out=0] (1.center) to (2.center);
		\draw [bend left=90, looseness=1.50] (3.center) to (2.center);
		\draw [bend right=90, looseness=1.50] (0.center) to (1.center);		
		\draw (0,0.25)  [fill=black] circle (\vr);
		\draw (3.5,0.25)  [fill=black] circle (\vr);
		\draw (7,0.25)  [fill=black] circle (\vr);
	\end{scope}
\begin{scope}[name=hedge,xshift=1.9cm, yshift=3.4cm]
		\def\vr{2pt}
		\node  (0) at (0, 0.5) {};
		\node  (1) at (0, 0) {};
		\node  (2) at (7, 0) {};
		\node (3) at (7, 0.5) {};
		\draw (0.center) to (3.center);
		\draw [in=180, out=0] (1.center) to (2.center);
		\draw [bend left=90, looseness=1.50] (3.center) to (2.center);
		\draw [bend right=90, looseness=1.50] (0.center) to (1.center);		
		\draw (0,0.25)  [fill=black] circle (\vr);
		\draw (3.5,0.25)  [fill=black] circle (\vr);
		\draw (7,0.25)  [fill=black] circle (\vr);
	\end{scope}
\end{scope}
	
	
		\begin{scope}[yshift = 9cm]

	\begin{scope}[rotate=60,xshift=0.211cm, yshift=-0.13cm]
		\def\vr{2pt}
		\node  (0) at (0, 0.5) {};
		\node  (1) at (0, 0) {};
		\node  (2) at (4, 0) {};
		\node (3) at (4, 0.5) {};
		\draw (0.center) to (3.center);
		\draw [in=180, out=0] (1.center) to (2.center);
		\draw [bend left=90, looseness=1.50] (3.center) to (2.center);
		\draw [bend right=90, looseness=1.50] (0.center) to (1.center);		
	\end{scope}
	
		\begin{scope}[rotate=60,xshift=1.9cm, yshift=-3.13cm]
		\def\vr{2pt}
		\node  (0) at (0, 0.5) {};
		\node  (1) at (0, 0) {};
		\node  (2) at (4, 0) {};
		\node (3) at (4, 0.5) {};
		\draw (0.center) to (3.center);
		\draw [in=180, out=0] (1.center) to (2.center);
		\draw [bend left=90, looseness=1.50] (3.center) to (2.center);
		\draw [bend right=90, looseness=1.50] (0.center) to (1.center);		
	\end{scope}
	
		\begin{scope}[rotate=60,xshift=3.63cm, yshift=-6.13cm]
		\def\vr{2pt}
		\node  (0) at (0, 0.5) {};
		\node  (1) at (0, 0) {};
		\node  (2) at (4, 0) {};
		\node (3) at (4, 0.5) {};
		\draw (0.center) to (3.center);
		\draw [in=180, out=0] (1.center) to (2.center);
		\draw [bend left=90, looseness=1.50] (3.center) to (2.center);
		\draw [bend right=90, looseness=1.50] (0.center) to (1.center);		
	\end{scope}

	\begin{scope}[name=hedge]
		\def\vr{2pt}
		\node  (0) at (0, 0.5) {};
		\node  (1) at (0, 0) {};
		\node  (2) at (7, 0) {};
		\node (3) at (7, 0.5) {};
		\draw (0.center) to (3.center);
		\draw [in=180, out=0] (1.center) to (2.center);
		\draw [bend left=90, looseness=1.50] (3.center) to (2.center);
		\draw [bend right=90, looseness=1.50] (0.center) to (1.center);		
		\draw (0,0.25)  [fill=black] circle (\vr);
		\draw (3.5,0.25)  [fill=black] circle (\vr);
		\draw (7,0.25)  [fill=black] circle (\vr);
	\end{scope}
	
	\begin{scope}[name=hedge,xshift=0.95cm, yshift=1.72cm]
		\def\vr{2pt}
		\node  (0) at (0, 0.5) {};
		\node  (1) at (0, 0) {};
		\node  (2) at (7, 0) {};
		\node (3) at (7, 0.5) {};
		\draw (0.center) to (3.center);
		\draw [in=180, out=0] (1.center) to (2.center);
		\draw [bend left=90, looseness=1.50] (3.center) to (2.center);
		\draw [bend right=90, looseness=1.50] (0.center) to (1.center);		
		\draw (0,0.25)  [fill=black] circle (\vr);
		\draw (3.5,0.25)  [fill=black] circle (\vr);
		\draw (7,0.25)  [fill=black] circle (\vr);
	\end{scope}
\begin{scope}[name=hedge,xshift=1.9cm, yshift=3.4cm]
		\def\vr{2pt}
		\node  (0) at (0, 0.5) {};
		\node  (1) at (0, 0) {};
		\node  (2) at (7, 0) {};
		\node (3) at (7, 0.5) {};
		\draw (0.center) to (3.center);
		\draw [in=180, out=0] (1.center) to (2.center);
		\draw [bend left=90, looseness=1.50] (3.center) to (2.center);
		\draw [bend right=90, looseness=1.50] (0.center) to (1.center);		
		\draw (0,0.25)  [fill=black] circle (\vr);
		\draw (3.5,0.25)  [fill=black] circle (\vr);
		\draw (7,0.25)  [fill=black] circle (\vr);
	\end{scope}
\end{scope}

\begin{scope}
	\begin{scope}[rotate=90,xshift=0.211cm, yshift=-0.27cm]
		\def\vr{2pt}
		\node  (0) at (0, 0.5) {};
		\node  (1) at (0, 0) {};
		\node  (2) at (9.1, 0) {};
		\node (3) at (9.1, 0.5) {};
		\draw (0.center) to (3.center);
		\draw [in=180, out=0] (1.center) to (2.center);
		\draw [bend left=90, looseness=1.50] (3.center) to (2.center);
		\draw [bend right=90, looseness=1.50] (0.center) to (1.center);		
	\end{scope}
	
		\begin{scope}[rotate=90,xshift=0.211cm, yshift=-3.7cm]
		\def\vr{2pt}
		\node  (0) at (0, 0.5) {};
		\node  (1) at (0, 0) {};
		\node  (2) at (9.1, 0) {};
		\node (3) at (9.1, 0.5) {};
		\draw (0.center) to (3.center);
		\draw [in=180, out=0] (1.center) to (2.center);
		\draw [bend left=90, looseness=1.50] (3.center) to (2.center);
		\draw [bend right=90, looseness=1.50] (0.center) to (1.center);		
	\end{scope}
	
			\begin{scope}[rotate=90,xshift=0.211cm, yshift=-7.2cm]
		\def\vr{2pt}
		\node  (0) at (0, 0.5) {};
		\node  (1) at (0, 0) {};
		\node  (2) at (9.1, 0) {};
		\node (3) at (9.1, 0.5) {};
		\draw (0.center) to (3.center);
		\draw [in=180, out=0] (1.center) to (2.center);
		\draw [bend left=90, looseness=1.50] (3.center) to (2.center);
		\draw [bend right=90, looseness=1.50] (0.center) to (1.center);		
	\end{scope}

\end{scope}

\begin{scope}[xshift=0.95cm, yshift=1.72cm]
	\begin{scope}[rotate=90,xshift=0.211cm, yshift=-0.27cm]
		\def\vr{2pt}
		\node  (0) at (0, 0.5) {};
		\node  (1) at (0, 0) {};
		\node  (2) at (9.1, 0) {};
		\node (3) at (9.1, 0.5) {};
		\draw (0.center) to (3.center);
		\draw [in=180, out=0] (1.center) to (2.center);
		\draw [bend left=90, looseness=1.50] (3.center) to (2.center);
		\draw [bend right=90, looseness=1.50] (0.center) to (1.center);		
	\end{scope}
	
		\begin{scope}[rotate=90,xshift=0.211cm, yshift=-3.7cm]
		\def\vr{2pt}
		\node  (0) at (0, 0.5) {};
		\node  (1) at (0, 0) {};
		\node  (2) at (9.1, 0) {};
		\node (3) at (9.1, 0.5) {};
		\draw (0.center) to (3.center);
		\draw [in=180, out=0] (1.center) to (2.center);
		\draw [bend left=90, looseness=1.50] (3.center) to (2.center);
		\draw [bend right=90, looseness=1.50] (0.center) to (1.center);		
	\end{scope}
	
			\begin{scope}[rotate=90,xshift=0.211cm, yshift=-7.2cm]
		\def\vr{2pt}
		\node  (0) at (0, 0.5) {};
		\node  (1) at (0, 0) {};
		\node  (2) at (9.1, 0) {};
		\node (3) at (9.1, 0.5) {};
		\draw (0.center) to (3.center);
		\draw [in=180, out=0] (1.center) to (2.center);
		\draw [bend left=90, looseness=1.50] (3.center) to (2.center);
		\draw [bend right=90, looseness=1.50] (0.center) to (1.center);		
	\end{scope}

\end{scope}

\begin{scope}[name=hedge,xshift=1.94cm, yshift=3.4cm]
	\begin{scope}[rotate=90,xshift=0.211cm, yshift=-0.27cm]
		\def\vr{2pt}
		\node  (0) at (0, 0.5) {};
		\node  (1) at (0, 0) {};
		\node  (2) at (9.1, 0) {};
		\node (3) at (9.1, 0.5) {};

		\draw (0.center) to (3.center);
		\draw [in=180, out=0] (1.center) to (2.center);
		\draw [bend left=90, looseness=1.50] (3.center) to (2.center);
		\draw [bend right=90, looseness=1.50] (0.center) to (1.center);		

	\end{scope}
	
		\begin{scope}[rotate=90,xshift=0.211cm, yshift=-3.7cm]
		\def\vr{2pt}
		\node  (0) at (0, 0.5) {};
		\node  (1) at (0, 0) {};
		\node  (2) at (9.1, 0) {};
		\node (3) at (9.1, 0.5) {};
		\draw (0.center) to (3.center);
		\draw [in=180, out=0] (1.center) to (2.center);
		\draw [bend left=90, looseness=1.50] (3.center) to (2.center);
		\draw [bend right=90, looseness=1.50] (0.center) to (1.center);		
	\end{scope}
	
			\begin{scope}[rotate=90,xshift=0.211cm, yshift=-7.2cm]
		\def\vr{2pt}
		\node  (0) at (0, 0.5) {};
		\node  (1) at (0, 0) {};
		\node  (2) at (9.1, 0) {};
		\node (3) at (9.1, 0.5) {};
		\draw (0.center) to (3.center);
		\draw [in=180, out=0] (1.center) to (2.center);
		\draw [bend left=90, looseness=1.50] (3.center) to (2.center);
		\draw [bend right=90, looseness=1.50] (0.center) to (1.center);		
	\end{scope}
\end{scope}
\end{tikzpicture}

\caption{Cube-hypergraphs}
\label{fig:cubes}
\end{figure}
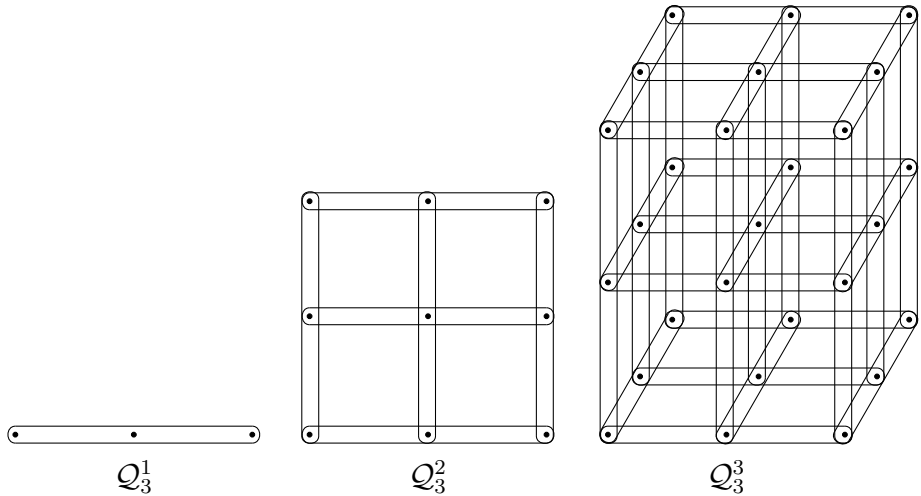

The $k$-uniform $n$-cube $\chg{k}{n}$ can be equivalently described as follows.
Its vertex set is $\{0,1, \ldots, k-1\}^n$ and an edge consists of all $n$-tuples which coincide on $n-1$ coordinates while the remaining coordinate ranges over $\{0, 1, \ldots, k-1\}$. It follows that $|V({\cal Q}_k^{n})| = k^n$ and $|E({\cal Q}_k^{n})| = n \cdot k^{n-1}$. Note that ${\cal Q}_2^{n}$ is a $2$-uniform hypergraph which is as a graph known as the $n$-cube.

\subsection{Partial cube-hypergraphs}

 A $k$-uniform hypergraph $H$ is a {\em partial cube-hypergraph} if $H$ is an isometric subhypergraph of some ${\cal Q}_k^{n}$.
 
A hypergraph $H$ is {\em edge-gated} if for any edge $e = \{a_1, \ldots, a_k\} \in E(H)$ and any vertex $x\in V(H)$ there exists $j \in [k]$ such that $d_H(x, a_i) = d_H(x, a_j) + 1$ for $i \in [k]$, $i\neq j$. We also say that $a_j$ is the {\em gate} of $x$ in $e$.
Note that if $x \in  e$ then $x$ is its own gate in $e$.

It is easy to see that in $2$-uniform hypergraphs (alias graphs) $H$ is edge-gated if and only if $H$ is bipartite. From this reason, edge-gated hypergraphs were named bipartite hypergraphs in \cite{burosch-1995}, where this concept was originally introduced. However, since there are numerous ways how bipartite graphs can be extended to hypergraphs we decided to change the terminology. The present terminology also mimics the established graph terminology, cf.~\cite{colbourn-2008}. 

It is easy to see that if hypergraphs $H$ and $H'$ are both edge-gated then so is $H \cp H'$. Also, if $H'$ is a connected isometric subgraph of an edge-gated hypergraph $H$, then $H'$ is edge-gated as well. It follows that partial cube-hypergraphs and hence in particular $k$-uniform $n$-cubes are edge-gated.

If $x$ and $y$ are two (adjacent) vertices of a hypergraph $H$, then let $H(x,y)$ denote the set of vertices that are closer to $x$ than to $y$, that is, 
$$H(x,y) = \{ z \in V(H) :\ d_H(z, x) < d_H(z, y)\}.$$
Further, if $e=\{a_1, \ldots, a_k\} \in E(H)$, then let $$H(a_i,e) = \{ z \in V(H) :\ d_H(z, a_i) < d_H(z, a_j),\ j\neq i \}. $$
In addition, set 
$$H_e = \{H(a_1,e), \ldots, H(a_k, e)\}.$$ 

Let now $H$ be an edge-gated hypergraph and $e=\{a_1, \ldots, a_k\} \in E(H)$. Since $a_i \in H(a_i, e)$ we have the following important facts.
\begin{lem}
\label{bip_prop}
{\rm \cite[Lemma~1(ii), Lemma~2]{burosch-1995}}
	If $H$ is an edge-gated hypergraph and $e = \{a_1, \ldots, a_k\} \in E(H)$, then the following statements hold.
	\begin{enumerate}[(i)]
		\item $H_e$ is a partition of $V(H)$.
		\item If $e' \in E(H)$, then either $|e' \cap H(a_i, e)| = 1$ for all $i \in [k]$   or $e' \subseteq H(a_i, e)$ for some $i \in [k]$.
	\end{enumerate}
	
\end{lem}

We next recall the following, key definition from~\cite{burosch-1995}. If $H$ is a hypergraph, then the binary relation $\Theta$ is defined on $E(H)$ as follows:
$$ e \Theta e' \  \equiv \  \forall A \in H_e :\ e'\cap A \neq \emptyset .$$
Note first that for any edge $e \in E(H)$ we have $e\Theta e$. If $H$ is edge-gated, then $\Theta$ is also symmetric by Lemma~\ref{bip_prop}(ii). Moreover, we recall the following important fact.
\begin{lem}
\label{lemma:22}
	{\rm \cite[Lemma~3]{burosch-1995}}
	If $H$ is an edge-gated hypergraph and for every $e\in E(H)$, every $A \in H_e$ is convex, then $f \Theta f'$ if and only if $H_f = H_{f'} $.

\end{lem}
For hypergraphs which fulfil the conditions of Lemma~\ref{lemma:22},  the relation $\Theta$ is an equivalence relation where the transitivity is guaranteed by Lemma~\ref{lemma:22}. Partial cube-hypergraphs which are $k$-uniform can now be characterized as follows.

\begin{thm}
\label{thm:1}
{\rm \cite[Theorem~1]{burosch-1995}} A $k$-uniform hypergraph $H$ is a partial cube-hypergraph if and only if $H$ is edge-gated and for every $e \in E(H)$, every $A \in H_e$ is convex.
\end{thm}

\begin{thm}
\label{thm:2}
{\rm \cite[Theorem~2]{burosch-1995}}
A $k$-uniform hypergraph $H$ is a partial cube-hypergraph if and only if $H$ is edge-gated and $\Theta$ is transitive.
\end{thm}

\section{Cut method for hypergraphs}
\label{sec:cut}
We now have all the tools needed for the main theorem of this article. But before we can formulate it, we need two additional auxiliary results and the following concepts.

If $H$ is a connected hypergraph, then $F \subseteq E(H)$ is a {\em cut} if the edges from $F$ are pairwise disjoint and $H-F$ consists of at least two components. We further say that the cut $F$ is a {\em convex cut} if the vertex set of each component of $H - F$ is a convex set.

Let $H$ be a $k$-uniform partial cube-hypergraph. Theorems~\ref{thm:1} and~\ref{thm:2} imply that the $\Theta$ relation is an equivalence relation on $E(H)$. We will denote its equivalence classes by $F_1, \ldots, F_m$. In addition, if $e \in E(H)$, then the equivalence class with the representative $e$ will also be denoted by $F_e$, that is, $F_e= \{f \in E(H) \  :\ e \Theta f \}$. 

From Lemma~\ref{lemma:22} we infer that the hypergraph $H - F_e$ consists of components  whose vertex sets are precisely the sets from $H_e$. This yields the following important fact.
\begin{prop}
\label{prop:1}
	Let $H $ be $k$-uniform partial cube-hypergraph and let $e \in E(H)$. Then $H - F_e$ has exactly $k$ components.
\end{prop} 

We also need the following auxiliary result.

\begin{prop}
\label{prop:2}
	Let $H$ be $k$-uniform partial cube-hypergraph and let $e \in E(H)$. If $u$ and $v$ are vertices from different components of $H - F_e$, then every shortest $u,v$-path contains exactly one edge from $F_e$.
\end{prop}

\begin{proof}
	 By Proposition~\ref{prop:1}, $H - F_e$ contains $k$ components which we denote by $H_1, \ldots H_k$. We may without loss of generality assume that $u \in H_1$ and $v \in H_k$. Furthermore, let $F_e = \{e_1, \ldots, e_{\ell}\}$. By Lemma~\ref{lemma:22}(ii) the vertices $u_i$ and $v_i$ defined as 
	 $$\{u_i\} = V(H_1) \cap e_i \quad \mbox{and} \quad \{v_i\} = V(H_k) \cap e_i$$
	 are well-defined for every $i \in [\ell]$. From the edge-gated property of $H$ it follows that $d_H(u_i, v) = d_H(v_i, v) + 1$. Since every $u,v$-path contains at least one of the vertices $u_i$, every shortest $u,v$-path contains exactly one of the edges $e_i$, $i \in [\ell]$.
\end{proof}

Let $H$ be a $k$-uniform partial cube-hypergraph and  let $F_1, \ldots, F_m$ be its $\Theta$-classes. By Proposition~\ref{prop:1}, $H - F_i$ has $k$ components, we denote them in the sequel by $ H_1(F_i), \ldots, H_k(F_i)$. Set in addition 
\begin{equation}
\label{nji}
	n_j(F_i) = |V(H_j(F_i))|, \ j \in [k], \ i \in [m].
\end{equation}

The cut method for hypergraphs now reads as follows.

\begin{thm}
\label{thm:cut_method}
	If $H$ is a  $k$-uniform partial cube-hypergraph, $F_1, \ldots, F_m$ are its $\Theta$-classes, and integers $n_j(F_i)$ are defined as in~\eqref{nji}, then 
	$$W(H) = \sum_{i=1}^m \sum_{ \{j, j'\} \in  {[k] \choose 2}} n_j(F_i) \cdot n_{j'}(F_i) \,.$$
\end{thm}
\begin{proof}
	Since $F_1, \ldots, F_m$ form a partition of $E(H)$, the idea is to consider the contribution of each edge to $W(H)$. Consider arbitrary vertices $u$ and $v$ of $H$ and an arbitrary $u,v$-shortest path $P$. By Proposition~\ref{prop:2}, edges from $P$ pairwise lie in different $\Theta$-classes of $E(H)$. If $e$ is an edge of $P$, then the contribution of $F_e$ to the distance $d_H(u,v)$ is exactly $1$. Consequently, the contribution of $F_e$ to $W(H)$ is exactly 
	$$\sum_{ \{j, j'\} \in  {[k] \choose 2}} n_j(F_e) \cdot n_{j'}(F_e).$$
	Summing over all $\Theta$-classes the result follows. 
		\end{proof}
\section{Some applications}
		
		In this section we give some examples and applications of Theorem~\ref{thm:cut_method}.
		\subsection{Cube-hypergraphs}
Cube-hypergraphs are partial cube-hypergraphs by definition. Hence Theorem~\ref{thm:cut_method} applies to them and leads to the following result.
\begin{prop} 
\label{prop:cube}
If $n \ge 1$ and $k \ge 2$, then
	$$ W({\cal Q}_k^{n}) = n {k \choose 2} k^{2(n-1)}.$$
\end{prop}
\begin{proof}
	To apply Theorem~\ref{thm:cut_method}, we first determine the $\Theta$-classes of ${\cal Q}_k^{n}$. Let an edge $e \in E(\chg{k}{n})$ be of the form $\{a_i = (i,0,\ldots,0) \ : \ i \in \{0,1, \ldots, k-1\}\}$. Then $H(e, a_i)$ contains the vertices $(i, v_2, \ldots, v_n)$, where $(v_2, \ldots, v_n) \in \{0,1,\ldots, k-1\}^{n-1}$. By Theorem~\ref{thm:1}, sets $H(e, a_i)$ are convex and the subhypergraphs induced by them are isomorphic to $\chg{k}{n-1}$. Then the $\Theta$-class $F_e = F_1$ contains all the edges whose last $n-1$ coordinates are fixed and the first coordinate ranges from $0$ to $k-1$. Using the same reasoning we get that every $\Theta$-class is of the above form. Therefore $\chg{k}{n}$ has $\Theta$-classes $F_1, \ldots, F_n$ where $\chg{k}{n} - F_i$ has components which are isomorphic to $\chg{k}{n-1}$ for $i \in [n]$. It then follows that $n_j(F_i) = k^{n-1}$ for every $j \in [k]$ and $i \in [n]$. From Theorem~\ref{thm:cut_method} it follows that
	\begin{align*}
		W(\chg{k}{n}) &= \sum_{i=1}^n \sum_{ \{j, j'\} \in  {[k] \choose 2}} k^{n-1}\cdot k^{n-1} = n \binom{k}{2}k^{2(n-1)},
	\end{align*}
	which we wanted to show.
	 
\end{proof}
Setting $k=2$, the hypergraph $\chg{2}{n}$ is the $n$-cube graph $Q_n$ and Proposition~\ref{prop:cube} implies a well-known result $W(Q_n) = n4^{n-1}$, which can in particular be deduced from the formula for the Wiener index of Cartesian products~\cite{graovac-1991}.

		\subsection{Hypertrees}
		\label{sec:trees}
		
		A hypergraph $T$ is a {\em hypertree} if it is connected, linear, and has no cycles. Here a {\em cycle} in a hypergraph is defined just as we defined a path except that the first and the last vertex from the corresponding sequence coincide. A hypertree which is linear and $k$-uniform is a partial cube-hypergraph where every edge $e$ is it own $\Theta$-class. Hence Theorem~\ref{thm:cut_method} as a special case yields the following result.
		
		\begin{cor}
		\label{cor:tree}
			If $T$ is a $k$-uniform hypertree, then
			$$W(T) = \sum_{e \in E(T)} \sum_{\{j, j'\} \in \binom{[k]}{2}} n_j(e) \cdot n_{j'}(e),  $$
			where $n_j(e) = n_j(F_e)$.
		\end{cor}
		
		Actually Corollary~\ref{cor:tree} holds also if we do not require that a hypertree is uniform. For this sake one just needs to reformulate Proposition~\ref{prop:1} such that its conclusion asserts that for any edge $e \in E(T)$, the hypergraph $T - e$ has exactly $|e|$ components. Moreover the second key auxiliary result, Proposition~\ref{prop:2}, also holds by the fact that in a hypertree there is a unique shortest path between two vertices. In this way Corollary~\ref{cor:tree} extends to
		
		\begin{thm}
		\label{thm:cut_method_tree}
		{\rm \cite[Theorem~3]{rodriguez-2005}}
			If $T$ is a hypertree, then
			$$W(T) = \sum_{e \in E(T)} \sum_{\{j, j'\} \in \binom{[|e|]}{2}} n_j(e) \cdot n_{j'}(e). $$
			
		\end{thm}
		
		For an example consider a hypertree $T_1$ from Figure~\ref{ex:tree1}. The hypertree $T_1$ has seven vertices and four edges. We now apply Theorem~\ref{thm:cut_method_tree}. For instance consider the edge $e = \{a_1,a_2,a_3\}$ as shown in the figure. Then $n_1(e) = 2$, $n_2(e) = 1$ and $n_3(e) = 4$. Therefore the contribution of $e$ to the formula of Theorem~\ref{thm:cut_method_tree} is $2\cdot 1 + 1\cdot 4 + 2\cdot 4$. Doing similar computations for the other three edges (see the bottom line of Fig.~\ref{ex:tree1}) we get
		$$ W(T_1) = 1\cdot 6 + (2\cdot 1 + 1\cdot 4 + 2\cdot 4) + (5 \cdot 1 + 5 \cdot  1+ 1 \cdot 1) + 6 \cdot 1 =37. $$

		\begin{figure}[ht!]
		\centering
		
			\begin{tikzpicture}[scale=0.45]
	\tikzstyle{nodee}=[fill=black, draw=black, shape=circle, scale=0.4]
	\begin{scope}
			\node [style=nodee] (0) at (-2.25, 0) {};
		\node [style=nodee] (1) at (-2.25, 1.75) {};
		\node [style=nodee] (2) at (-0.75, 0) {};
		\node [style=nodee] (3) at (0.75, 0) {};
		\node [style=nodee] (4) at (0.75, 1.5) {};
		\node [style=nodee] (5) at (0.75, 3) {};
		\node [style=nodee] (6) at (2.25, 0) {};
		\node [] (7) at (-2.75, 2) {};
		\node [] (8) at (-2.5, -0.5) {};
		\node [] (9) at (1, -0.5) {};
		\node [] (10) at (1, 0.5) {};
		\node [] (11) at (-1.75, 2) {};
		\node [] (12) at (-2.75, -0.25) {};
		\node [] (13) at (-1.75, -0.25) {};
		\node [] (14) at (-2.5, 0.5) {};
		\node [] (15) at (0.25, 3.25) {};
		\node [] (16) at (1.25, 3.25) {};
		\node [] (17) at (0.25, -0.25) {};
		\node [] (18) at (1.25, -0.25) {};
		\node [] (19) at (0.5, -0.5) {};
		\node [] (20) at (2.5, -0.5) {};
		\node [] (21) at (2.5, 0.5) {};
		\node [] (22) at (0.5, 0.5) {};

		\draw [, bend right, looseness=0.50] (7.center) to (12.center);
		\draw [, bend right=60, looseness=1.25] (12.center) to (13.center);
		\draw [, bend left, looseness=0.50] (11.center) to (13.center);
		\draw [, in=120, out=60, looseness=1.25] (7.center) to (11.center);
		\draw [, bend left=45, looseness=0.25] (14.center) to (10.center);
		\draw [, bend right, looseness=0.25] (8.center) to (9.center);
		\draw [, bend left=75, looseness=1.25] (10.center) to (9.center);
		\draw [, bend right=75, looseness=1.50] (14.center) to (8.center);
		\draw [, bend right, looseness=0.50] (15.center) to (17.center);
		\draw [, bend right=75, looseness=1.25] (17.center) to (18.center);
		\draw [, in=60, out=-60, looseness=0.50] (16.center) to (18.center);
		\draw [, bend left=75, looseness=1.50] (15.center) to (16.center);
		\draw [, in=165, out=15, looseness=0.75] (22.center) to (21.center);
		\draw [, bend right=15, looseness=0.75] (19.center) to (20.center);
		\draw [, bend left=75, looseness=1.50] (21.center) to (20.center);
		\draw [, bend right=75, looseness=1.25] (22.center) to (19.center);
		\node []  at (-0.7, -1.2) {$e$};

	\end{scope}
	
	\begin{scope}	[xshift=-12cm, yshift=-6cm]	
	\node [style=nodee] (0) at (-2.25, 0) {};
		\node [style=nodee] (1) at (-2.25, 1.75) {};
		\node [style=nodee] (2) at (-0.75, 0) {};
		\node [style=nodee] (3) at (0.75, 0) {};
		\node [style=nodee] (4) at (0.75, 1.5) {};
		\node [style=nodee] (5) at (0.75, 3) {};
		\node [style=nodee] (6) at (2.25, 0) {};
		\node [] (7) at (-2.75, 2) {};
		\node [] (8) at (-2.5, -0.5) {};
		\node [] (9) at (1, -0.5) {};
		\node [] (10) at (1, 0.5) {};
		\node [] (11) at (-1.75, 2) {};
		\node [] (12) at (-2.75, -0.25) {};
		\node [] (13) at (-1.75, -0.25) {};
		\node [] (14) at (-2.5, 0.5) {};
		\node [] (15) at (0.25, 3.25) {};
		\node [] (16) at (1.25, 3.25) {};
		\node [] (17) at (0.25, -0.25) {};
		\node [] (18) at (1.25, -0.25) {};
		\node [] (19) at (0.5, -0.5) {};
		\node [] (20) at (2.5, -0.5) {};
		\node [] (21) at (2.5, 0.5) {};
		\node [] (22) at (0.5, 0.5) {};

		\draw [, bend left=45, looseness=0.25] (14.center) to (10.center);
		\draw [, bend right, looseness=0.25] (8.center) to (9.center);
		\draw [, bend left=75, looseness=1.25] (10.center) to (9.center);
		\draw [, bend right=75, looseness=1.50] (14.center) to (8.center);
		\draw [, bend right, looseness=0.50] (15.center) to (17.center);
		\draw [, bend right=75, looseness=1.25] (17.center) to (18.center);
		\draw [, in=60, out=-60, looseness=0.50] (16.center) to (18.center);
		\draw [, bend left=75, looseness=1.50] (15.center) to (16.center);
		\draw [, in=165, out=15, looseness=0.75] (22.center) to (21.center);
		\draw [, bend right=15, looseness=0.75] (19.center) to (20.center);
		\draw [, bend left=75, looseness=1.50] (21.center) to (20.center);
		\draw [, bend right=75, looseness=1.25] (22.center) to (19.center);

	\end{scope}
		\begin{scope}	[xshift=-4cm, yshift=-6cm]	
	\node [style=nodee] (0) at (-2.25, 0) {};
		\node [style=nodee] (1) at (-2.25, 1.75) {};
		\node [style=nodee] (2) at (-0.75, 0) {};
		\node [style=nodee] (3) at (0.75, 0) {};
		\node [style=nodee] (4) at (0.75, 1.5) {};
		\node [style=nodee] (5) at (0.75, 3) {};
		\node [style=nodee] (6) at (2.25, 0) {};
		\node [] (7) at (-2.75, 2) {};
		\node [] (8) at (-2.5, -0.5) {};
		\node [] (9) at (1, -0.5) {};
		\node [] (10) at (1, 0.5) {};
		\node [] (11) at (-1.75, 2) {};
		\node [] (12) at (-2.75, -0.25) {};
		\node [] (13) at (-1.75, -0.25) {};
		\node [] (14) at (-2.5, 0.5) {};
		\node [] (15) at (0.25, 3.25) {};
		\node [] (16) at (1.25, 3.25) {};
		\node [] (17) at (0.25, -0.25) {};
		\node [] (18) at (1.25, -0.25) {};
		\node [] (19) at (0.5, -0.5) {};
		\node [] (20) at (2.5, -0.5) {};
		\node [] (21) at (2.5, 0.5) {};
		\node [] (22) at (0.5, 0.5) {};

		\draw [, bend right, looseness=0.50] (7.center) to (12.center);
		\draw [, bend right=60, looseness=1.25] (12.center) to (13.center);
		\draw [, bend left, looseness=0.50] (11.center) to (13.center);
		\draw [, in=120, out=60, looseness=1.25] (7.center) to (11.center);
		\draw [, bend right, looseness=0.50] (15.center) to (17.center);
		\draw [, bend right=75, looseness=1.25] (17.center) to (18.center);
		\draw [, in=60, out=-60, looseness=0.50] (16.center) to (18.center);
		\draw [, bend left=75, looseness=1.50] (15.center) to (16.center);
		\draw [, in=165, out=15, looseness=0.75] (22.center) to (21.center);
		\draw [, bend right=15, looseness=0.75] (19.center) to (20.center);
		\draw [, bend left=75, looseness=1.50] (21.center) to (20.center);
		\draw [, bend right=75, looseness=1.25] (22.center) to (19.center);
		
		 \node []  at (-2.2, -1.2) {$a_1$};
		 \node []  at (-0.7, -1.2) {$a_2$};

 		\node []  at (0.7, -1.2) {$a_3$};

	\end{scope}
	
		\begin{scope}	[xshift=4cm, yshift=-6cm]	
	\node [style=nodee] (0) at (-2.25, 0) {};
		\node [style=nodee] (1) at (-2.25, 1.75) {};
		\node [style=nodee] (2) at (-0.75, 0) {};
		\node [style=nodee] (3) at (0.75, 0) {};
		\node [style=nodee] (4) at (0.75, 1.5) {};
		\node [style=nodee] (5) at (0.75, 3) {};
		\node [style=nodee] (6) at (2.25, 0) {};
		\node [] (7) at (-2.75, 2) {};
		\node [] (8) at (-2.5, -0.5) {};
		\node [] (9) at (1, -0.5) {};
		\node [] (10) at (1, 0.5) {};
		\node [] (11) at (-1.75, 2) {};
		\node [] (12) at (-2.75, -0.25) {};
		\node [] (13) at (-1.75, -0.25) {};
		\node [] (14) at (-2.5, 0.5) {};
		\node [] (15) at (0.25, 3.25) {};
		\node [] (16) at (1.25, 3.25) {};
		\node [] (17) at (0.25, -0.25) {};
		\node [] (18) at (1.25, -0.25) {};
		\node [] (19) at (0.5, -0.5) {};
		\node [] (20) at (2.5, -0.5) {};
		\node [] (21) at (2.5, 0.5) {};
		\node [] (22) at (0.5, 0.5) {};

		\draw [, bend right, looseness=0.50] (7.center) to (12.center);
		\draw [, bend right=60, looseness=1.25] (12.center) to (13.center);
		\draw [, bend left, looseness=0.50] (11.center) to (13.center);
		\draw [, in=120, out=60, looseness=1.25] (7.center) to (11.center);
		\draw [, bend left=45, looseness=0.25] (14.center) to (10.center);
		\draw [, bend right, looseness=0.25] (8.center) to (9.center);
		\draw [, bend left=75, looseness=1.25] (10.center) to (9.center);
		\draw [, bend right=75, looseness=1.50] (14.center) to (8.center);
		\draw [, in=165, out=15, looseness=0.75] (22.center) to (21.center);
		\draw [, bend right=15, looseness=0.75] (19.center) to (20.center);
		\draw [, bend left=75, looseness=1.50] (21.center) to (20.center);
		\draw [, bend right=75, looseness=1.25] (22.center) to (19.center);

	\end{scope}
	
		\begin{scope}	[xshift=12cm, yshift=-6cm]	
	\node [style=nodee] (0) at (-2.25, 0) {};
		\node [style=nodee] (1) at (-2.25, 1.75) {};
		\node [style=nodee] (2) at (-0.75, 0) {};
		\node [style=nodee] (3) at (0.75, 0) {};
		\node [style=nodee] (4) at (0.75, 1.5) {};
		\node [style=nodee] (5) at (0.75, 3) {};
		\node [style=nodee] (6) at (2.25, 0) {};
		\node [] (7) at (-2.75, 2) {};
		\node [] (8) at (-2.5, -0.5) {};
		\node [] (9) at (1, -0.5) {};
		\node [] (10) at (1, 0.5) {};
		\node [] (11) at (-1.75, 2) {};
		\node [] (12) at (-2.75, -0.25) {};
		\node [] (13) at (-1.75, -0.25) {};
		\node [] (14) at (-2.5, 0.5) {};
		\node [] (15) at (0.25, 3.25) {};
		\node [] (16) at (1.25, 3.25) {};
		\node [] (17) at (0.25, -0.25) {};
		\node [] (18) at (1.25, -0.25) {};
		\node [] (19) at (0.5, -0.5) {};
		\node [] (20) at (2.5, -0.5) {};
		\node [] (21) at (2.5, 0.5) {};
		\node [] (22) at (0.5, 0.5) {};

		\draw [, bend right, looseness=0.50] (7.center) to (12.center);
		\draw [, bend right=60, looseness=1.25] (12.center) to (13.center);
		\draw [, bend left, looseness=0.50] (11.center) to (13.center);
		\draw [, in=120, out=60, looseness=1.25] (7.center) to (11.center);
		\draw [, bend left=45, looseness=0.25] (14.center) to (10.center);
		\draw [, bend right, looseness=0.25] (8.center) to (9.center);
		\draw [, bend left=75, looseness=1.25] (10.center) to (9.center);
		\draw [, bend right=75, looseness=1.50] (14.center) to (8.center);
		\draw [, bend right, looseness=0.50] (15.center) to (17.center);
		\draw [, bend right=75, looseness=1.25] (17.center) to (18.center);
		\draw [, in=60, out=-60, looseness=0.50] (16.center) to (18.center);
		\draw [, bend left=75, looseness=1.50] (15.center) to (16.center);

	\end{scope}

\end{tikzpicture}

\caption{Hypertree $T_1$.}
\label{ex:tree1}
\end{figure}
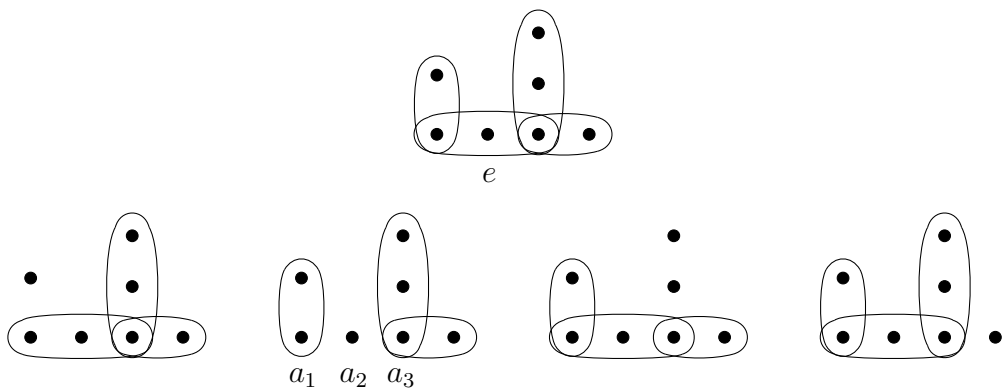

		A limitation of Theorem~\ref{thm:cut_method_tree} is that it only works for linear hypertrees. On the other hand, there exist many different definitions of {\em acyclicity} in hypergraphs, where some of them also allow for non-linear hypergraphs. See for example~\cite{jegou-2009}. We next show with an example that the main idea of Theorem~\ref{thm:cut_method_tree} can sometimes be generalized to such cases as well. 
		
		Define the {\em linear phenylene} hypergraphs $LP_n$, $n
		\ge 2$, as follows. (For some recent studies of phenylenes in mathematical chemistry see~\cite{chen-2019, knor-2023, li-2020a, liu-2022}.) $LP_n$ has vertex set $[6n]$. It has $2n - 1$ hyperedges. The first $n$ of them are of the form $\{6i + 1, 6i + 2, \ldots, 6i + 6\}$ where $i \in \{0,1, \ldots, n-1\}$, and the remaining $n-1$ hyperedges edges are of the form $\{6i + 5, 6i + 6, 6i + 7, 6i + 8\}$, where $i \in \{0,1,\ldots, n-2\}$. In Figure~\ref{ex:phenylene} the hypergraph $LP_4$ is drawn.
		
		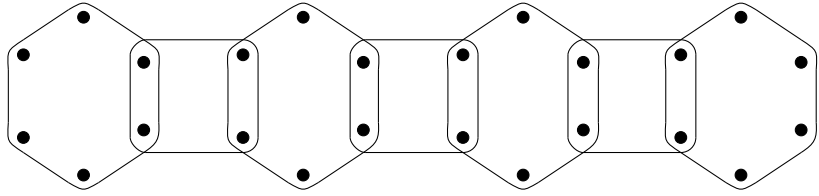
\begin{figure}[ht!]
		\centering
			\begin{tikzpicture}[scale=0.4]
			\tikzstyle{nodee}=[fill=black, draw=black, shape=circle, scale=0.4]
			
			\begin{scope}[rotate=90]

		\node  (0) at (-0.75, 0) {};
		\node  (1) at (-1.75, 0.5) {};
		\node  (2) at (1, 0) {};
		\node  (3) at (2, 0.5) {};
		\node  (4) at (3, 2) {};
		\node  (5) at (3, 3) {};
		\node  (6) at (-2.75, 2) {};
		\node  (7) at (-2.75, 3) {};
		\node  (8) at (-1.75, 4.5) {};
		\node  (9) at (-0.75, 5) {};
		\node (10) at (1, 5) {};
		\node  (11) at (2, 4.5) {};
		\node [style=nodee] (12) at (-1, 0.5) {};
		\node [style=nodee] (13) at (1.25, 0.5) {};
		\node [style=nodee] (14) at (2.75, 2.5) {};
		\node [style=nodee] (15) at (1.5, 4.5) {};
		\node [style=nodee] (16) at (-1.25, 4.5) {};
		\node [style=nodee] (17) at (-2.5, 2.5) {};

		\draw (0.center) to (2.center);
		\draw [bend right, looseness=1.50] (2.center) to (3.center);
		\draw [bend left, looseness=1.25] (0.center) to (1.center);
		\draw (6.center) to (1.center);
		\draw (3.center) to (4.center);
		\draw [bend right, looseness=1.50] (4.center) to (5.center);
		\draw [bend left, looseness=1.50] (6.center) to (7.center);
		\draw (8.center) to (7.center);
		\draw (9.center) to (10.center);
		\draw (5.center) to (11.center);
		\draw [bend left, looseness=1.50] (10.center) to (11.center);
		\draw [bend left, looseness=1.50] (8.center) to (9.center);

\end{scope}

\begin{scope}[rotate=90,yshift=-4.8cm]
		\node [] (0) at (-0.75, -2.5) {};nodee
		\node [] (1) at (-1.75, -2) {};
		\node [] (2) at (1, -2.5) {};
		\node [] (3) at (2, -2) {};
		\node [] (4) at (3, -0.5) {};
		\node [] (5) at (3, 0.5) {};
		\node [] (6) at (-2.75, -0.5) {};
		\node [] (7) at (-2.75, 0.5) {};
		\node [] (8) at (-1.75, 2) {};
		\node [] (9) at (-0.75, 2.5) {};
		\node [] (10) at (1, 2.5) {};
		\node [] (11) at (2, 2) {};
		\node [style=nodee] (12) at (-1, -2) {};
		\node [style=nodee] (13) at (1.25, -2) {};
		\node [style=nodee] (14) at (2.75, 0) {};
		\node [style=nodee] (15) at (1.5, 2) {};
		\node [style=nodee] (16) at (-1.25, 2) {};
		\node [style=nodee] (17) at (-2.5, 0) {};
		\node [] (18) at (-1.75, 2) {};
		\node [] (21) at (2, 2) {};
		\node [] (22) at (-1.75, 5.25) {};
		\node [] (25) at (2, 5.25) {};
		\node [] (26) at (-1.25, 5.75) {};
		\node [] (27) at (1.5, 5.75) {};
		\node [] (28) at (-1.25, 1.5) {};
		\node [] (29) at (1.5, 1.5) {};

		\draw (0.center) to (2.center);
		\draw [bend right, looseness=1.50] (2.center) to (3.center);
		\draw [bend left, looseness=1.25] (0.center) to (1.center);
		\draw (6.center) to (1.center);
		\draw (3.center) to (4.center);
		\draw [bend right, looseness=1.50] (4.center) to (5.center);
		\draw [bend left, looseness=1.50] (6.center) to (7.center);
		\draw (8.center) to (7.center);
		\draw (9.center) to (10.center);
		\draw (5.center) to (11.center);
		\draw [bend left, looseness=1.50] (10.center) to (11.center);
		\draw [bend left, looseness=1.50] (8.center) to (9.center);
		\draw [] (22.center) to (18.center);
		\draw [] (25.center) to (21.center);
		\draw [ bend left=45, looseness=0.75] (22.center) to (26.center);
		\draw [ bend left=45, looseness=0.75] (27.center) to (25.center);
		\draw [] (26.center) to (27.center);
		\draw [ bend right=45] (18.center) to (28.center);
		\draw [ bend left=45] (21.center) to (29.center);
		\draw [] (28.center) to (29.center);

\end{scope}

\begin{scope}[rotate=90,yshift=-12.11cm]
		\node [] (0) at (-0.75, -2.5) {};nodee
		\node [] (1) at (-1.75, -2) {};
		\node [] (2) at (1, -2.5) {};
		\node [] (3) at (2, -2) {};
		\node [] (4) at (3, -0.5) {};
		\node [] (5) at (3, 0.5) {};
		\node [] (6) at (-2.75, -0.5) {};
		\node [] (7) at (-2.75, 0.5) {};
		\node [] (8) at (-1.75, 2) {};
		\node [] (9) at (-0.75, 2.5) {};
		\node [] (10) at (1, 2.5) {};
		\node [] (11) at (2, 2) {};
		\node [style=nodee] (12) at (-1, -2) {};
		\node [style=nodee] (13) at (1.25, -2) {};
		\node [style=nodee] (14) at (2.75, 0) {};
		\node [style=nodee] (15) at (1.5, 2) {};
		\node [style=nodee] (16) at (-1.25, 2) {};
		\node [style=nodee] (17) at (-2.5, 0) {};
		\node [] (18) at (-1.75, 2) {};
		\node [] (21) at (2, 2) {};
		\node [] (22) at (-1.75, 5.25) {};
		\node [] (25) at (2, 5.25) {};
		\node [] (26) at (-1.25, 5.75) {};
		\node [] (27) at (1.5, 5.75) {};
		\node [] (28) at (-1.25, 1.5) {};
		\node [] (29) at (1.5, 1.5) {};

		\draw (0.center) to (2.center);
		\draw [bend right, looseness=1.50] (2.center) to (3.center);
		\draw [bend left, looseness=1.25] (0.center) to (1.center);
		\draw (6.center) to (1.center);
		\draw (3.center) to (4.center);
		\draw [bend right, looseness=1.50] (4.center) to (5.center);
		\draw [bend left, looseness=1.50] (6.center) to (7.center);
		\draw (8.center) to (7.center);
		\draw (9.center) to (10.center);
		\draw (5.center) to (11.center);
		\draw [bend left, looseness=1.50] (10.center) to (11.center);
		\draw [bend left, looseness=1.50] (8.center) to (9.center);
		\draw [] (22.center) to (18.center);
		\draw [] (25.center) to (21.center);
		\draw [ bend left=45, looseness=0.75] (22.center) to (26.center);
		\draw [ bend left=45, looseness=0.75] (27.center) to (25.center);
		\draw [] (26.center) to (27.center);
		\draw [ bend right=45] (18.center) to (28.center);
		\draw [ bend left=45] (21.center) to (29.center);
		\draw [] (28.center) to (29.center);

\end{scope}

\begin{scope}[rotate=90,yshift=-19.35cm]
		\node  (0) at (-0.75, -2.5) {};
		\node  (1) at (-1.75, -2) {};
		\node  (2) at (1, -2.5) {};
		\node [] (3) at (2, -2) {};
		\node [] (4) at (3, -0.5) {};
		\node [] (5) at (3, 0.5) {};
		\node [] (6) at (-2.75, -0.5) {};
		\node [] (7) at (-2.75, 0.5) {};
		\node [] (8) at (-1.75, 2) {};
		\node [] (9) at (-0.75, 2.5) {};
		\node [] (10) at (1, 2.5) {};
		\node [] (11) at (2, 2) {};
		\node [style=nodee] (12) at (-1, -2) {};
		\node [style=nodee] (13) at (1.25, -2) {};
		\node [style=nodee] (14) at (2.75, 0) {};
		\node [style=nodee] (15) at (1.5, 2) {};
		\node [style=nodee] (16) at (-1.25, 2) {};
		\node [style=nodee] (17) at (-2.5, 0) {};
		\node [] (18) at (-1.75, 2) {};
		\node [] (21) at (2, 2) {};
		\node [] (22) at (-1.75, 5.25) {};
		\node [] (25) at (2, 5.25) {};
		\node [] (26) at (-1.25, 5.75) {};
		\node [] (27) at (1.5, 5.75) {};
		\node [] (28) at (-1.25, 1.5) {};
		\node [] (29) at (1.5, 1.5) {};

		\draw (0.center) to (2.center);
		\draw [bend right, looseness=1.50] (2.center) to (3.center);
		\draw [bend left, looseness=1.25] (0.center) to (1.center);
		\draw (6.center) to (1.center);
		\draw (3.center) to (4.center);
		\draw [bend right, looseness=1.50] (4.center) to (5.center);
		\draw [bend left, looseness=1.50] (6.center) to (7.center);
		\draw (8.center) to (7.center);
		\draw (9.center) to (10.center);
		\draw (5.center) to (11.center);
		\draw [bend left, looseness=1.50] (10.center) to (11.center);
		\draw [bend left, looseness=1.50] (8.center) to (9.center);
		\draw [] (22.center) to (18.center);
		\draw [] (25.center) to (21.center);
		\draw [ bend left=45, looseness=0.75] (22.center) to (26.center);
		\draw [ bend left=45, looseness=0.75] (27.center) to (25.center);
		\draw [] (26.center) to (27.center);
		\draw [ bend right=45] (18.center) to (28.center);
		\draw [ bend left=45] (21.center) to (29.center);
		\draw [] (28.center) to (29.center);

\end{scope}

\end{tikzpicture}

\caption{Hypergraph $LP_4$.}
\label{ex:phenylene}
		\end{figure}
		
		 It is easy to see that every edge $e \in E(LP_n)$ is a convex cut with the following property. Taking any two vertices $u,v$ from different components of $LP_n - e$, every shortest $u,v$-path contains $e$ (exactly once). Note, however, that the two vertices which lie in the intersection of two hyperedges are not separated by any of the cuts. But it is clear that the distance between such two vertices is $1$. Together there are $2(n-1)$ such pairs and therefore this number needs to be added to the Wiener index of $LP_n$. This is enough to calculate the Wiener index of $LP_n$ using cut method as follows. 
		 
		 Removing an edge of the form $\{6i + 1, 6i + 2, \ldots, 6i + 6\}$, where $i \in [n-2]$, produces four components where two of them contain a single vertex and the remaining two have $6i + 2$ and $6(n-i-1) + 2$ vertices, respectively. The cases when $i=0$ or $i = n-1$ give five components each,  four of them contain a single vertex, while the remaining one contains $6n - 4$ vertices. On the other hand, removing an edge of the form $\{6i + 5, 6i + 6, 6i + 7, 6i + 8\}$ produces two components with $6(i+1)$ and $6(n-i-1)$ vertices, respectively. Therefore, the contribution of all these cuts to the Wiener index of $LP_n$ for $n>1$  is
		\begin{align*}
				& \sum_{i=1}^{n-2} \left[ 2(6i + 2 + 6(n - i - 1) + 2) + (6i+2)(6n-6i-4) +1 \right]  \\
				+& 2 \left(\binom{4}{2} + 4 (6n - 4) \right)    \\
				 +&\sum_{i=0}^{n-2}\left[6(i+1)\cdot 6(n-i-1)\right]=  12n^3 + 6n^2 - 5n + 2,
		\end{align*}
		where the second line above comes from the contribution of the first hyperedge and the last hyperedge containing six vertices. Adding to this expression the contribution $2(n-1)$	 from previous paragraph and performing a straightforward computation we arrive to the following result. 
		\begin{prop}
			If $n \ge 2$ then, $ W(LP_n) = 12n^3 + 6n^2 - 3n$. 
		\end{prop}

		\subsection{More elaborate example}
		The cut method as developed in Section~\ref{sec:cut} assumes that a hypergraph is a $k$-uniform partial cube-hypergraph. In general this is a strong assumption. We have just demonstrated in Section~\ref{sec:trees} that the method can be extended also when the hypergraph is not $k$-uniform partial cube-hypergraph, provided that Propositions~\ref{prop:1} and~\ref{prop:2} remain valid. In the subsequent example we further elaborate this idea on a mulecular hypergraph $H$ of a Clar structure which is shown in Figure~\ref{fig:t}(a) and in~\cite[Fig. 3]{gutman-1999}.	
			
		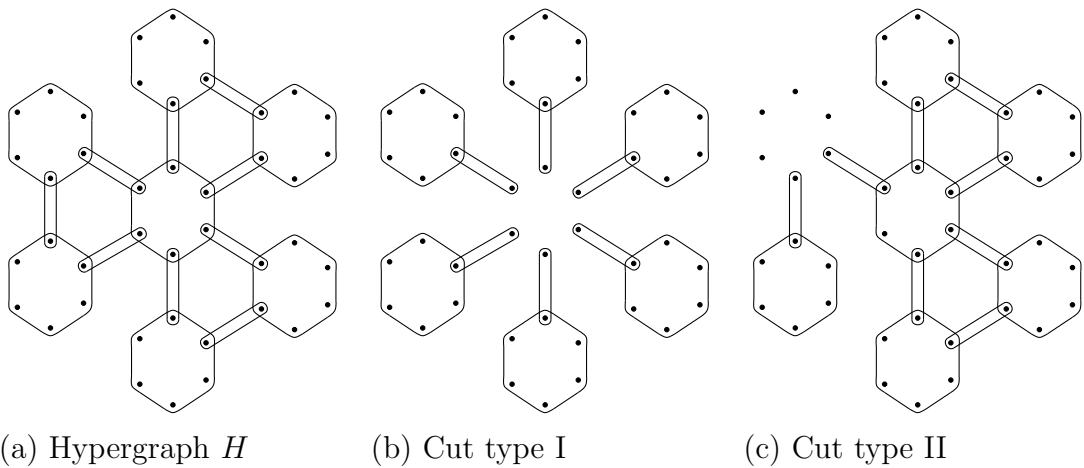
\begin{figure}[ht!]
		
		\centering
				\minipage{0.33\textwidth}
\begin{tikzpicture}[scale=0.22]
\tikzstyle{nodee}=[fill=black, draw=black, shape=circle, scale=0.15]
\begin{scope}[rotate=90]

		\node  (0) at (-0.75, 0) {};
		\node  (1) at (-1.75, 0.5) {};
		\node  (2) at (1, 0) {};
		\node  (3) at (2, 0.5) {};
		\node  (4) at (3, 2) {};
		\node  (5) at (3, 3) {};
		\node  (6) at (-2.75, 2) {};
		\node  (7) at (-2.75, 3) {};
		\node  (8) at (-1.75, 4.5) {};
		\node  (9) at (-0.75, 5) {};
		\node (10) at (1, 5) {};
		\node  (11) at (2, 4.5) {};
		\node [style=nodee] (12) at (-1, 0.5) {};
		\node [style=nodee] (13) at (1.25, 0.5) {};
		\node [style=nodee] (14) at (2.75, 2.5) {};
		\node [style=nodee] (15) at (1.5, 4.5) {};
		\node [style=nodee] (16) at (-1.25, 4.5) {};
		\node [style=nodee] (17) at (-2.5, 2.5) {};

		\draw (0.center) to (2.center);
		\draw [bend right, looseness=1.50] (2.center) to (3.center);
		\draw [bend left, looseness=1.25] (0.center) to (1.center);
		\draw (6.center) to (1.center);
		\draw (3.center) to (4.center);
		\draw [bend right, looseness=1.50] (4.center) to (5.center);
		\draw [bend left, looseness=1.50] (6.center) to (7.center);
		\draw (8.center) to (7.center);
		\draw (9.center) to (10.center);
		\draw (5.center) to (11.center);
		\draw [bend left, looseness=1.50] (10.center) to (11.center);
		\draw [bend left, looseness=1.50] (8.center) to (9.center);

\end{scope}

\begin{scope}
	\begin{scope}[name=hedge, rotate=-32, xshift=0.1cm, yshift=-1.45cm]
		\def\vr{2pt}
		\node  (0) at (0, 0.7) {};
		\node  (1) at (0, 0) {};
		\node  (2) at (4, 0) {};
		\node (3) at (4, 0.7) {};
		\draw (0.center) to (3.center);
		\draw [in=180, out=0] (1.center) to (2.center);
		\draw [bend left=90, looseness=1.50] (3.center) to (2.center);
		\draw [bend right=90, looseness=1.50] (0.center) to (1.center);		
	\end{scope}

	\begin{scope}[name=hedge, rotate=-90, xshift=2.45cm, yshift=-2.85cm]
		\def\vr{2pt}
		\node  (0) at (0, 0.7) {};
		\node  (1) at (0, 0) {};
		\node  (2) at (4, 0) {};
		\node (3) at (4, 0.7) {};
		\draw (0.center) to (3.center);
		\draw [in=180, out=0] (1.center) to (2.center);
		\draw [bend left=90, looseness=1.50] (3.center) to (2.center);
		\draw [bend right=90, looseness=1.50] (0.center) to (1.center);		
	\end{scope}	
	
	\begin{scope}[name=hedge, rotate=32, xshift=-4.6cm, yshift=-6.75cm]
		\def\vr{2pt}
		\node  (0) at (0, 0.7) {};
		\node  (1) at (0, 0) {};
		\node  (2) at (4, 0) {};
		\node (3) at (4, 0.7) {};
		\draw (0.center) to (3.center);
		\draw [in=180, out=0] (1.center) to (2.center);
		\draw [bend left=90, looseness=1.50] (3.center) to (2.center);
		\draw [bend right=90, looseness=1.50] (0.center) to (1.center);		
	\end{scope}	
	\begin{scope}[rotate=90,xshift=-9.1cm]
		
		\node  (0) at (-0.75, 0) {};
		\node  (1) at (-1.75, 0.5) {};
		\node  (2) at (1, 0) {};
		\node  (3) at (2, 0.5) {};
		\node  (4) at (3, 2) {};
		\node  (5) at (3, 3) {};
		\node  (6) at (-2.75, 2) {};
		\node  (7) at (-2.75, 3) {};
		\node  (8) at (-1.75, 4.5) {};
		\node  (9) at (-0.75, 5) {};
		\node (10) at (1, 5) {};
		\node  (11) at (2, 4.5) {};
		\node [style=nodee] (12) at (-1, 0.5) {};
		\node [style=nodee] (13) at (1.25, 0.5) {};
		\node [style=nodee] (14) at (2.75, 2.5) {};
		\node [style=nodee] (15) at (1.5, 4.5) {};
		\node [style=nodee] (16) at (-1.25, 4.5) {};
		\node [style=nodee] (17) at (-2.5, 2.5) {};

		\draw (0.center) to (2.center);
		\draw [bend right, looseness=1.50] (2.center) to (3.center);
		\draw [bend left, looseness=1.25] (0.center) to (1.center);
		\draw (6.center) to (1.center);
		\draw (3.center) to (4.center);
		\draw [bend right, looseness=1.50] (4.center) to (5.center);
		\draw [bend left, looseness=1.50] (6.center) to (7.center);
		\draw (8.center) to (7.center);
		\draw (9.center) to (10.center);
		\draw (5.center) to (11.center);
		\draw [bend left, looseness=1.50] (10.center) to (11.center);
		\draw [bend left, looseness=1.50] (8.center) to (9.center);
\end{scope}

\begin{scope}[rotate=90,xshift=-4.55cm, yshift=-7.35cm]
	
		\node  (0) at (-0.75, 0) {};
		\node  (1) at (-1.75, 0.5) {};
		\node  (2) at (1, 0) {};
		\node  (3) at (2, 0.5) {};
		\node  (4) at (3, 2) {};
		\node  (5) at (3, 3) {};
		\node  (6) at (-2.75, 2) {};
		\node  (7) at (-2.75, 3) {};
		\node  (8) at (-1.75, 4.5) {};
		\node  (9) at (-0.75, 5) {};
		\node (10) at (1, 5) {};
		\node  (11) at (2, 4.5) {};
		\node [style=nodee] (12) at (-1, 0.5) {};
		\node [style=nodee] (13) at (1.25, 0.5) {};
		\node [style=nodee] (14) at (2.75, 2.5) {};
		\node [style=nodee] (15) at (1.5, 4.5) {};
		\node [style=nodee] (16) at (-1.25, 4.5) {};
		\node [style=nodee] (17) at (-2.5, 2.5) {};

		\draw (0.center) to (2.center);
		\draw [bend right, looseness=1.50] (2.center) to (3.center);
		\draw [bend left, looseness=1.25] (0.center) to (1.center);
		\draw (6.center) to (1.center);
		\draw (3.center) to (4.center);
		\draw [bend right, looseness=1.50] (4.center) to (5.center);
		\draw [bend left, looseness=1.50] (6.center) to (7.center);
		\draw (8.center) to (7.center);
		\draw (9.center) to (10.center);
		\draw (5.center) to (11.center);
		\draw [bend left, looseness=1.50] (10.center) to (11.center);
		\draw [bend left, looseness=1.50] (8.center) to (9.center);
\end{scope}
\end{scope}

\begin{scope}[yscale=-1, yshift=-0.26cm]
	\begin{scope}[name=hedge, rotate=-32, xshift=0.1cm, yshift=-1.45cm]
		\def\vr{2pt}
		\node  (0) at (0, 0.7) {};
		\node  (1) at (0, 0) {};
		\node  (2) at (4, 0) {};
		\node (3) at (4, 0.7) {};
		\draw (0.center) to (3.center);
		\draw [in=180, out=0] (1.center) to (2.center);
		\draw [bend left=90, looseness=1.50] (3.center) to (2.center);
		\draw [bend right=90, looseness=1.50] (0.center) to (1.center);		
	\end{scope}

	\begin{scope}[name=hedge, rotate=-90, xshift=2.45cm, yshift=-2.85cm]
		\def\vr{2pt}
		\node  (0) at (0, 0.7) {};
		\node  (1) at (0, 0) {};
		\node  (2) at (4, 0) {};
		\node (3) at (4, 0.7) {};
		\draw (0.center) to (3.center);
		\draw [in=180, out=0] (1.center) to (2.center);
		\draw [bend left=90, looseness=1.50] (3.center) to (2.center);
		\draw [bend right=90, looseness=1.50] (0.center) to (1.center);		
	\end{scope}	
	
	\begin{scope}[name=hedge, rotate=32, xshift=-4.6cm, yshift=-6.75cm]
		\def\vr{2pt}
		\node  (0) at (0, 0.7) {};
		\node  (1) at (0, 0) {};
		\node  (2) at (4, 0) {};
		\node (3) at (4, 0.7) {};
		\draw (0.center) to (3.center);
		\draw [in=180, out=0] (1.center) to (2.center);
		\draw [bend left=90, looseness=1.50] (3.center) to (2.center);
		\draw [bend right=90, looseness=1.50] (0.center) to (1.center);		
	\end{scope}	
	\begin{scope}[rotate=90,xshift=-9.1cm]
		
		\node  (0) at (-0.75, 0) {};
		\node  (1) at (-1.75, 0.5) {};
		\node  (2) at (1, 0) {};
		\node  (3) at (2, 0.5) {};
		\node  (4) at (3, 2) {};
		\node  (5) at (3, 3) {};
		\node  (6) at (-2.75, 2) {};
		\node  (7) at (-2.75, 3) {};
		\node  (8) at (-1.75, 4.5) {};
		\node  (9) at (-0.75, 5) {};
		\node (10) at (1, 5) {};
		\node  (11) at (2, 4.5) {};
		\node [style=nodee] (12) at (-1, 0.5) {};
		\node [style=nodee] (13) at (1.25, 0.5) {};
		\node [style=nodee] (14) at (2.75, 2.5) {};
		\node [style=nodee] (15) at (1.5, 4.5) {};
		\node [style=nodee] (16) at (-1.25, 4.5) {};
		\node [style=nodee] (17) at (-2.5, 2.5) {};

		\draw (0.center) to (2.center);
		\draw [bend right, looseness=1.50] (2.center) to (3.center);
		\draw [bend left, looseness=1.25] (0.center) to (1.center);
		\draw (6.center) to (1.center);
		\draw (3.center) to (4.center);
		\draw [bend right, looseness=1.50] (4.center) to (5.center);
		\draw [bend left, looseness=1.50] (6.center) to (7.center);
		\draw (8.center) to (7.center);
		\draw (9.center) to (10.center);
		\draw (5.center) to (11.center);
		\draw [bend left, looseness=1.50] (10.center) to (11.center);
		\draw [bend left, looseness=1.50] (8.center) to (9.center);
\end{scope}

\begin{scope}[rotate=90,xshift=-4.55cm, yshift=-7.35cm]
	
		\node  (0) at (-0.75, 0) {};
		\node  (1) at (-1.75, 0.5) {};
		\node  (2) at (1, 0) {};
		\node  (3) at (2, 0.5) {};
		\node  (4) at (3, 2) {};
		\node  (5) at (3, 3) {};
		\node  (6) at (-2.75, 2) {};
		\node  (7) at (-2.75, 3) {};
		\node  (8) at (-1.75, 4.5) {};
		\node  (9) at (-0.75, 5) {};
		\node (10) at (1, 5) {};
		\node  (11) at (2, 4.5) {};
		\node [style=nodee] (12) at (-1, 0.5) {};
		\node [style=nodee] (13) at (1.25, 0.5) {};
		\node [style=nodee] (14) at (2.75, 2.5) {};
		\node [style=nodee] (15) at (1.5, 4.5) {};
		\node [style=nodee] (16) at (-1.25, 4.5) {};
		\node [style=nodee] (17) at (-2.5, 2.5) {};

		\draw (0.center) to (2.center);
		\draw [bend right, looseness=1.50] (2.center) to (3.center);
		\draw [bend left, looseness=1.25] (0.center) to (1.center);
		\draw (6.center) to (1.center);
		\draw (3.center) to (4.center);
		\draw [bend right, looseness=1.50] (4.center) to (5.center);
		\draw [bend left, looseness=1.50] (6.center) to (7.center);
		\draw (8.center) to (7.center);
		\draw (9.center) to (10.center);
		\draw (5.center) to (11.center);
		\draw [bend left, looseness=1.50] (10.center) to (11.center);
		\draw [bend left, looseness=1.50] (8.center) to (9.center);
\end{scope}
\end{scope}

\begin{scope}
\begin{scope}[name=hedge, rotate=-32, xshift=-8.6cm, yshift=-1.45cm]
		\def\vr{2pt}
		\node  (0) at (0, 0.7) {};
		\node  (1) at (0, 0) {};
		\node  (2) at (4, 0) {};
		\node (3) at (4, 0.7) {};
		\draw (0.center) to (3.center);
		\draw [in=180, out=0] (1.center) to (2.center);
		\draw [bend left=90, looseness=1.50] (3.center) to (2.center);
		\draw [bend right=90, looseness=1.50] (0.center) to (1.center);		
	\end{scope}

	\begin{scope}[name=hedge, rotate=-150, xshift=4.45cm, yshift=-1.5cm]
		\def\vr{2pt}
		\node  (0) at (0, 0.7) {};
		\node  (1) at (0, 0) {};
		\node  (2) at (4, 0) {};
		\node (3) at (4, 0.7) {};
		\draw (0.center) to (3.center);
		\draw [in=180, out=0] (1.center) to (2.center);
		\draw [bend left=90, looseness=1.50] (3.center) to (2.center);
		\draw [bend right=90, looseness=1.50] (0.center) to (1.center);		
	\end{scope}	
	
	\begin{scope}[name=hedge, rotate=90, xshift=-1.8cm, yshift=9.55cm]
		\def\vr{2pt}
		\node  (0) at (0, 0.7) {};
		\node  (1) at (0, 0) {};
		\node  (2) at (4, 0) {};
		\node (3) at (4, 0.7) {};
		\draw (0.center) to (3.center);
		\draw [in=180, out=0] (1.center) to (2.center);
		\draw [bend left=90, looseness=1.50] (3.center) to (2.center);
		\draw [bend right=90, looseness=1.50] (0.center) to (1.center);		
	\end{scope}	
	
	\begin{scope}[rotate=90,xshift=-4.45cm, yshift=7.4cm]
			\node  (0) at (-0.75, 0) {};
		\node  (1) at (-1.75, 0.5) {};
		\node  (2) at (1, 0) {};
		\node  (3) at (2, 0.5) {};
		\node  (4) at (3, 2) {};
		\node  (5) at (3, 3) {};
		\node  (6) at (-2.75, 2) {};
		\node  (7) at (-2.75, 3) {};
		\node  (8) at (-1.75, 4.5) {};
		\node  (9) at (-0.75, 5) {};
		\node (10) at (1, 5) {};
		\node  (11) at (2, 4.5) {};
		\node [style=nodee] (12) at (-1, 0.5) {};
		\node [style=nodee] (13) at (1.25, 0.5) {};
		\node [style=nodee] (14) at (2.75, 2.5) {};
		\node [style=nodee] (15) at (1.5, 4.5) {};
		\node [style=nodee] (16) at (-1.25, 4.5) {};
		\node [style=nodee] (17) at (-2.5, 2.5) {};

		\draw (0.center) to (2.center);
		\draw [bend right, looseness=1.50] (2.center) to (3.center);
		\draw [bend left, looseness=1.25] (0.center) to (1.center);
		\draw (6.center) to (1.center);
		\draw (3.center) to (4.center);
		\draw [bend right, looseness=1.50] (4.center) to (5.center);
		\draw [bend left, looseness=1.50] (6.center) to (7.center);
		\draw (8.center) to (7.center);
		\draw (9.center) to (10.center);
		\draw (5.center) to (11.center);
		\draw [bend left, looseness=1.50] (10.center) to (11.center);
		\draw [bend left, looseness=1.50] (8.center) to (9.center);
\end{scope}

	\begin{scope}[rotate=90,xshift=4.6cm, yshift=7.4cm]
			\node  (0) at (-0.75, 0) {};
		\node  (1) at (-1.75, 0.5) {};
		\node  (2) at (1, 0) {};
		\node  (3) at (2, 0.5) {};
		\node  (4) at (3, 2) {};
		\node  (5) at (3, 3) {};
		\node  (6) at (-2.75, 2) {};
		\node  (7) at (-2.75, 3) {};
		\node  (8) at (-1.75, 4.5) {};
		\node  (9) at (-0.75, 5) {};
		\node (10) at (1, 5) {};
		\node  (11) at (2, 4.5) {};
		\node [style=nodee] (12) at (-1, 0.5) {};
		\node [style=nodee] (13) at (1.25, 0.5) {};
		\node [style=nodee] (14) at (2.75, 2.5) {};
		\node [style=nodee] (15) at (1.5, 4.5) {};
		\node [style=nodee] (16) at (-1.25, 4.5) {};
		\node [style=nodee] (17) at (-2.5, 2.5) {};

		\draw (0.center) to (2.center);
		\draw [bend right, looseness=1.50] (2.center) to (3.center);
		\draw [bend left, looseness=1.25] (0.center) to (1.center);
		\draw (6.center) to (1.center);
		\draw (3.center) to (4.center);
		\draw [bend right, looseness=1.50] (4.center) to (5.center);
		\draw [bend left, looseness=1.50] (6.center) to (7.center);
		\draw (8.center) to (7.center);
		\draw (9.center) to (10.center);
		\draw (5.center) to (11.center);
		\draw [bend left, looseness=1.50] (10.center) to (11.center);
		\draw [bend left, looseness=1.50] (8.center) to (9.center);
\end{scope}

\end{scope}

	\end{tikzpicture}

	 (a) Hypergraph $H$
	\endminipage
			\minipage{0.33\textwidth}
\begin{tikzpicture}[scale=0.22]
\tikzstyle{nodee}=[fill=black, draw=black, shape=circle, scale=0.15]
\begin{scope}[rotate=90]

		\node  (0) at (-0.75, 0) {};
		\node  (1) at (-1.75, 0.5) {};
		\node  (2) at (1, 0) {};
		\node  (3) at (2, 0.5) {};
		\node  (4) at (3, 2) {};
		\node  (5) at (3, 3) {};
		\node  (6) at (-2.75, 2) {};
		\node  (7) at (-2.75, 3) {};
		\node  (8) at (-1.75, 4.5) {};
		\node  (9) at (-0.75, 5) {};
		\node (10) at (1, 5) {};
		\node  (11) at (2, 4.5) {};
		\node [style=nodee] (12) at (-1, 0.5) {};
		\node [style=nodee] (13) at (1.25, 0.5) {};
		\node [style=nodee] (14) at (2.75, 2.5) {};
		\node [style=nodee] (15) at (1.5, 4.5) {};
		\node [style=nodee] (16) at (-1.25, 4.5) {};
		\node [style=nodee] (17) at (-2.5, 2.5) {};


\end{scope}

\begin{scope}
	\begin{scope}[name=hedge, rotate=-32, xshift=0.1cm, yshift=-1.45cm]
		\def\vr{2pt}
		\node  (0) at (0, 0.7) {};
		\node  (1) at (0, 0) {};
		\node  (2) at (4, 0) {};
		\node (3) at (4, 0.7) {};
		\draw (0.center) to (3.center);
		\draw [in=180, out=0] (1.center) to (2.center);
		\draw [bend left=90, looseness=1.50] (3.center) to (2.center);
		\draw [bend right=90, looseness=1.50] (0.center) to (1.center);		
	\end{scope}

	\begin{scope}[name=hedge, rotate=-90, xshift=2.45cm, yshift=-2.85cm]
		\def\vr{2pt}
		\node  (0) at (0, 0.7) {};
		\node  (1) at (0, 0) {};
		\node  (2) at (4, 0) {};
		\node (3) at (4, 0.7) {};
		\draw (0.center) to (3.center);
		\draw [in=180, out=0] (1.center) to (2.center);
		\draw [bend left=90, looseness=1.50] (3.center) to (2.center);
		\draw [bend right=90, looseness=1.50] (0.center) to (1.center);		
	\end{scope}	
	
	\begin{scope}[name=hedge, rotate=32, xshift=-4.6cm, yshift=-6.75cm]
		\def\vr{2pt}
		\node  (0) at (0, 0.7) {};
		\node  (1) at (0, 0) {};
		\node  (2) at (4, 0) {};
		\node (3) at (4, 0.7) {};
	\end{scope}	
	\begin{scope}[rotate=90,xshift=-9.1cm]
		
		\node  (0) at (-0.75, 0) {};
		\node  (1) at (-1.75, 0.5) {};
		\node  (2) at (1, 0) {};
		\node  (3) at (2, 0.5) {};
		\node  (4) at (3, 2) {};
		\node  (5) at (3, 3) {};
		\node  (6) at (-2.75, 2) {};
		\node  (7) at (-2.75, 3) {};
		\node  (8) at (-1.75, 4.5) {};
		\node  (9) at (-0.75, 5) {};
		\node (10) at (1, 5) {};
		\node  (11) at (2, 4.5) {};
		\node [style=nodee] (12) at (-1, 0.5) {};
		\node [style=nodee] (13) at (1.25, 0.5) {};
		\node [style=nodee] (14) at (2.75, 2.5) {};
		\node [style=nodee] (15) at (1.5, 4.5) {};
		\node [style=nodee] (16) at (-1.25, 4.5) {};
		\node [style=nodee] (17) at (-2.5, 2.5) {};

		\draw (0.center) to (2.center);
		\draw [bend right, looseness=1.50] (2.center) to (3.center);
		\draw [bend left, looseness=1.25] (0.center) to (1.center);
		\draw (6.center) to (1.center);
		\draw (3.center) to (4.center);
		\draw [bend right, looseness=1.50] (4.center) to (5.center);
		\draw [bend left, looseness=1.50] (6.center) to (7.center);
		\draw (8.center) to (7.center);
		\draw (9.center) to (10.center);
		\draw (5.center) to (11.center);
		\draw [bend left, looseness=1.50] (10.center) to (11.center);
		\draw [bend left, looseness=1.50] (8.center) to (9.center);
\end{scope}

\begin{scope}[rotate=90,xshift=-4.55cm, yshift=-7.35cm]
	
		\node  (0) at (-0.75, 0) {};
		\node  (1) at (-1.75, 0.5) {};
		\node  (2) at (1, 0) {};
		\node  (3) at (2, 0.5) {};
		\node  (4) at (3, 2) {};
		\node  (5) at (3, 3) {};
		\node  (6) at (-2.75, 2) {};
		\node  (7) at (-2.75, 3) {};
		\node  (8) at (-1.75, 4.5) {};
		\node  (9) at (-0.75, 5) {};
		\node (10) at (1, 5) {};
		\node  (11) at (2, 4.5) {};
		\node [style=nodee] (12) at (-1, 0.5) {};
		\node [style=nodee] (13) at (1.25, 0.5) {};
		\node [style=nodee] (14) at (2.75, 2.5) {};
		\node [style=nodee] (15) at (1.5, 4.5) {};
		\node [style=nodee] (16) at (-1.25, 4.5) {};
		\node [style=nodee] (17) at (-2.5, 2.5) {};

		\draw (0.center) to (2.center);
		\draw [bend right, looseness=1.50] (2.center) to (3.center);
		\draw [bend left, looseness=1.25] (0.center) to (1.center);
		\draw (6.center) to (1.center);
		\draw (3.center) to (4.center);
		\draw [bend right, looseness=1.50] (4.center) to (5.center);
		\draw [bend left, looseness=1.50] (6.center) to (7.center);
		\draw (8.center) to (7.center);
		\draw (9.center) to (10.center);
		\draw (5.center) to (11.center);
		\draw [bend left, looseness=1.50] (10.center) to (11.center);
		\draw [bend left, looseness=1.50] (8.center) to (9.center);
\end{scope}
\end{scope}

\begin{scope}[yscale=-1, yshift=-0.26cm]
	\begin{scope}[name=hedge, rotate=-32, xshift=0.1cm, yshift=-1.45cm]
		\def\vr{2pt}
		\node  (0) at (0, 0.7) {};
		\node  (1) at (0, 0) {};
		\node  (2) at (4, 0) {};
		\node (3) at (4, 0.7) {};
		\draw (0.center) to (3.center);
		\draw [in=180, out=0] (1.center) to (2.center);
		\draw [bend left=90, looseness=1.50] (3.center) to (2.center);
		\draw [bend right=90, looseness=1.50] (0.center) to (1.center);		
	\end{scope}

	\begin{scope}[name=hedge, rotate=-90, xshift=2.45cm, yshift=-2.85cm]
		\def\vr{2pt}
		\node  (0) at (0, 0.7) {};
		\node  (1) at (0, 0) {};
		\node  (2) at (4, 0) {};
		\node (3) at (4, 0.7) {};
		\draw (0.center) to (3.center);
		\draw [in=180, out=0] (1.center) to (2.center);
		\draw [bend left=90, looseness=1.50] (3.center) to (2.center);
		\draw [bend right=90, looseness=1.50] (0.center) to (1.center);		
	\end{scope}	
	
	\begin{scope}[name=hedge, rotate=32, xshift=-4.6cm, yshift=-6.75cm]
		\def\vr{2pt}
		\node  (0) at (0, 0.7) {};
		\node  (1) at (0, 0) {};
		\node  (2) at (4, 0) {};
		\node (3) at (4, 0.7) {};
	\end{scope}	
	\begin{scope}[rotate=90,xshift=-9.1cm]
		
		\node  (0) at (-0.75, 0) {};
		\node  (1) at (-1.75, 0.5) {};
		\node  (2) at (1, 0) {};
		\node  (3) at (2, 0.5) {};
		\node  (4) at (3, 2) {};
		\node  (5) at (3, 3) {};
		\node  (6) at (-2.75, 2) {};
		\node  (7) at (-2.75, 3) {};
		\node  (8) at (-1.75, 4.5) {};
		\node  (9) at (-0.75, 5) {};
		\node (10) at (1, 5) {};
		\node  (11) at (2, 4.5) {};
		\node [style=nodee] (12) at (-1, 0.5) {};
		\node [style=nodee] (13) at (1.25, 0.5) {};
		\node [style=nodee] (14) at (2.75, 2.5) {};
		\node [style=nodee] (15) at (1.5, 4.5) {};
		\node [style=nodee] (16) at (-1.25, 4.5) {};
		\node [style=nodee] (17) at (-2.5, 2.5) {};

		\draw (0.center) to (2.center);
		\draw [bend right, looseness=1.50] (2.center) to (3.center);
		\draw [bend left, looseness=1.25] (0.center) to (1.center);
		\draw (6.center) to (1.center);
		\draw (3.center) to (4.center);
		\draw [bend right, looseness=1.50] (4.center) to (5.center);
		\draw [bend left, looseness=1.50] (6.center) to (7.center);
		\draw (8.center) to (7.center);
		\draw (9.center) to (10.center);
		\draw (5.center) to (11.center);
		\draw [bend left, looseness=1.50] (10.center) to (11.center);
		\draw [bend left, looseness=1.50] (8.center) to (9.center);
\end{scope}

\begin{scope}[rotate=90,xshift=-4.55cm, yshift=-7.35cm]
	
		\node  (0) at (-0.75, 0) {};
		\node  (1) at (-1.75, 0.5) {};
		\node  (2) at (1, 0) {};
		\node  (3) at (2, 0.5) {};
		\node  (4) at (3, 2) {};
		\node  (5) at (3, 3) {};
		\node  (6) at (-2.75, 2) {};
		\node  (7) at (-2.75, 3) {};
		\node  (8) at (-1.75, 4.5) {};
		\node  (9) at (-0.75, 5) {};
		\node (10) at (1, 5) {};
		\node  (11) at (2, 4.5) {};
		\node [style=nodee] (12) at (-1, 0.5) {};
		\node [style=nodee] (13) at (1.25, 0.5) {};
		\node [style=nodee] (14) at (2.75, 2.5) {};
		\node [style=nodee] (15) at (1.5, 4.5) {};
		\node [style=nodee] (16) at (-1.25, 4.5) {};
		\node [style=nodee] (17) at (-2.5, 2.5) {};

		\draw (0.center) to (2.center);
		\draw [bend right, looseness=1.50] (2.center) to (3.center);
		\draw [bend left, looseness=1.25] (0.center) to (1.center);
		\draw (6.center) to (1.center);
		\draw (3.center) to (4.center);
		\draw [bend right, looseness=1.50] (4.center) to (5.center);
		\draw [bend left, looseness=1.50] (6.center) to (7.center);
		\draw (8.center) to (7.center);
		\draw (9.center) to (10.center);
		\draw (5.center) to (11.center);
		\draw [bend left, looseness=1.50] (10.center) to (11.center);
		\draw [bend left, looseness=1.50] (8.center) to (9.center);
\end{scope}
\end{scope}

\begin{scope}
\begin{scope}[name=hedge, rotate=-32, xshift=-8.6cm, yshift=-1.45cm]
		\def\vr{2pt}
		\node  (0) at (0, 0.7) {};
		\node  (1) at (0, 0) {};
		\node  (2) at (4, 0) {};
		\node (3) at (4, 0.7) {};
		\draw (0.center) to (3.center);
		\draw [in=180, out=0] (1.center) to (2.center);
		\draw [bend left=90, looseness=1.50] (3.center) to (2.center);
		\draw [bend right=90, looseness=1.50] (0.center) to (1.center);		
	\end{scope}

	\begin{scope}[name=hedge, rotate=-150, xshift=4.45cm, yshift=-1.5cm]
		\def\vr{2pt}
		\node  (0) at (0, 0.7) {};
		\node  (1) at (0, 0) {};
		\node  (2) at (4, 0) {};
		\node (3) at (4, 0.7) {};
		\draw (0.center) to (3.center);
		\draw [in=180, out=0] (1.center) to (2.center);
		\draw [bend left=90, looseness=1.50] (3.center) to (2.center);
		\draw [bend right=90, looseness=1.50] (0.center) to (1.center);		
	\end{scope}	
	
	\begin{scope}[name=hedge, rotate=90, xshift=-1.8cm, yshift=9.55cm]
		\def\vr{2pt}
		\node  (0) at (0, 0.7) {};
		\node  (1) at (0, 0) {};
		\node  (2) at (4, 0) {};
		\node (3) at (4, 0.7) {};
	\end{scope}	
	
	\begin{scope}[rotate=90,xshift=-4.45cm, yshift=7.4cm]
			\node  (0) at (-0.75, 0) {};
		\node  (1) at (-1.75, 0.5) {};
		\node  (2) at (1, 0) {};
		\node  (3) at (2, 0.5) {};
		\node  (4) at (3, 2) {};
		\node  (5) at (3, 3) {};
		\node  (6) at (-2.75, 2) {};
		\node  (7) at (-2.75, 3) {};
		\node  (8) at (-1.75, 4.5) {};
		\node  (9) at (-0.75, 5) {};
		\node (10) at (1, 5) {};
		\node  (11) at (2, 4.5) {};
		\node [style=nodee] (12) at (-1, 0.5) {};
		\node [style=nodee] (13) at (1.25, 0.5) {};
		\node [style=nodee] (14) at (2.75, 2.5) {};
		\node [style=nodee] (15) at (1.5, 4.5) {};
		\node [style=nodee] (16) at (-1.25, 4.5) {};
		\node [style=nodee] (17) at (-2.5, 2.5) {};

		\draw (0.center) to (2.center);
		\draw [bend right, looseness=1.50] (2.center) to (3.center);
		\draw [bend left, looseness=1.25] (0.center) to (1.center);
		\draw (6.center) to (1.center);
		\draw (3.center) to (4.center);
		\draw [bend right, looseness=1.50] (4.center) to (5.center);
		\draw [bend left, looseness=1.50] (6.center) to (7.center);
		\draw (8.center) to (7.center);
		\draw (9.center) to (10.center);
		\draw (5.center) to (11.center);
		\draw [bend left, looseness=1.50] (10.center) to (11.center);
		\draw [bend left, looseness=1.50] (8.center) to (9.center);
\end{scope}

	\begin{scope}[rotate=90,xshift=4.6cm, yshift=7.4cm]
			\node  (0) at (-0.75, 0) {};
		\node  (1) at (-1.75, 0.5) {};
		\node  (2) at (1, 0) {};
		\node  (3) at (2, 0.5) {};
		\node  (4) at (3, 2) {};
		\node  (5) at (3, 3) {};
		\node  (6) at (-2.75, 2) {};
		\node  (7) at (-2.75, 3) {};
		\node  (8) at (-1.75, 4.5) {};
		\node  (9) at (-0.75, 5) {};
		\node (10) at (1, 5) {};
		\node  (11) at (2, 4.5) {};
		\node [style=nodee] (12) at (-1, 0.5) {};
		\node [style=nodee] (13) at (1.25, 0.5) {};
		\node [style=nodee] (14) at (2.75, 2.5) {};
		\node [style=nodee] (15) at (1.5, 4.5) {};
		\node [style=nodee] (16) at (-1.25, 4.5) {};
		\node [style=nodee] (17) at (-2.5, 2.5) {};

		\draw (0.center) to (2.center);
		\draw [bend right, looseness=1.50] (2.center) to (3.center);
		\draw [bend left, looseness=1.25] (0.center) to (1.center);
		\draw (6.center) to (1.center);
		\draw (3.center) to (4.center);
		\draw [bend right, looseness=1.50] (4.center) to (5.center);
		\draw [bend left, looseness=1.50] (6.center) to (7.center);
		\draw (8.center) to (7.center);
		\draw (9.center) to (10.center);
		\draw (5.center) to (11.center);
		\draw [bend left, looseness=1.50] (10.center) to (11.center);
		\draw [bend left, looseness=1.50] (8.center) to (9.center);
\end{scope}

\end{scope}

	\end{tikzpicture}

 (b) Cut type I
	\endminipage
		\minipage{0.33\textwidth}
\begin{tikzpicture}[scale=0.22]
\tikzstyle{nodee}=[fill=black, draw=black, shape=circle, scale=0.15]
\begin{scope}[rotate=90]

		\node  (0) at (-0.75, 0) {};
		\node  (1) at (-1.75, 0.5) {};
		\node  (2) at (1, 0) {};
		\node  (3) at (2, 0.5) {};
		\node  (4) at (3, 2) {};
		\node  (5) at (3, 3) {};
		\node  (6) at (-2.75, 2) {};
		\node  (7) at (-2.75, 3) {};
		\node  (8) at (-1.75, 4.5) {};
		\node  (9) at (-0.75, 5) {};
		\node (10) at (1, 5) {};
		\node  (11) at (2, 4.5) {};
		\node [style=nodee] (12) at (-1, 0.5) {};
		\node [style=nodee] (13) at (1.25, 0.5) {};
		\node [style=nodee] (14) at (2.75, 2.5) {};
		\node [style=nodee] (15) at (1.5, 4.5) {};
		\node [style=nodee] (16) at (-1.25, 4.5) {};
		\node [style=nodee] (17) at (-2.5, 2.5) {};

		\draw (0.center) to (2.center);
		\draw [bend right, looseness=1.50] (2.center) to (3.center);
		\draw [bend left, looseness=1.25] (0.center) to (1.center);
		\draw (6.center) to (1.center);
		\draw (3.center) to (4.center);
		\draw [bend right, looseness=1.50] (4.center) to (5.center);
		\draw [bend left, looseness=1.50] (6.center) to (7.center);
		\draw (8.center) to (7.center);
		\draw (9.center) to (10.center);
		\draw (5.center) to (11.center);
		\draw [bend left, looseness=1.50] (10.center) to (11.center);
		\draw [bend left, looseness=1.50] (8.center) to (9.center);

\end{scope}

\begin{scope}
	\begin{scope}[name=hedge, rotate=-32, xshift=0.1cm, yshift=-1.45cm]
		\def\vr{2pt}
		\node  (0) at (0, 0.7) {};
		\node  (1) at (0, 0) {};
		\node  (2) at (4, 0) {};
		\node (3) at (4, 0.7) {};
		\draw (0.center) to (3.center);
		\draw [in=180, out=0] (1.center) to (2.center);
		\draw [bend left=90, looseness=1.50] (3.center) to (2.center);
		\draw [bend right=90, looseness=1.50] (0.center) to (1.center);		
	\end{scope}

	\begin{scope}[name=hedge, rotate=-90, xshift=2.45cm, yshift=-2.85cm]
		\def\vr{2pt}
		\node  (0) at (0, 0.7) {};
		\node  (1) at (0, 0) {};
		\node  (2) at (4, 0) {};
		\node (3) at (4, 0.7) {};
		\draw (0.center) to (3.center);
		\draw [in=180, out=0] (1.center) to (2.center);
		\draw [bend left=90, looseness=1.50] (3.center) to (2.center);
		\draw [bend right=90, looseness=1.50] (0.center) to (1.center);		
	\end{scope}	
	
	\begin{scope}[name=hedge, rotate=32, xshift=-4.6cm, yshift=-6.75cm]
		\def\vr{2pt}
		\node  (0) at (0, 0.7) {};
		\node  (1) at (0, 0) {};
		\node  (2) at (4, 0) {};
		\node (3) at (4, 0.7) {};
		\draw (0.center) to (3.center);
		\draw [in=180, out=0] (1.center) to (2.center);
		\draw [bend left=90, looseness=1.50] (3.center) to (2.center);
		\draw [bend right=90, looseness=1.50] (0.center) to (1.center);		
	\end{scope}	
	\begin{scope}[rotate=90,xshift=-9.1cm]
		
		\node  (0) at (-0.75, 0) {};
		\node  (1) at (-1.75, 0.5) {};
		\node  (2) at (1, 0) {};
		\node  (3) at (2, 0.5) {};
		\node  (4) at (3, 2) {};
		\node  (5) at (3, 3) {};
		\node  (6) at (-2.75, 2) {};
		\node  (7) at (-2.75, 3) {};
		\node  (8) at (-1.75, 4.5) {};
		\node  (9) at (-0.75, 5) {};
		\node (10) at (1, 5) {};
		\node  (11) at (2, 4.5) {};
		\node [style=nodee] (12) at (-1, 0.5) {};
		\node [style=nodee] (13) at (1.25, 0.5) {};
		\node [style=nodee] (14) at (2.75, 2.5) {};
		\node [style=nodee] (15) at (1.5, 4.5) {};
		\node [style=nodee] (16) at (-1.25, 4.5) {};
		\node [style=nodee] (17) at (-2.5, 2.5) {};

		\draw (0.center) to (2.center);
		\draw [bend right, looseness=1.50] (2.center) to (3.center);
		\draw [bend left, looseness=1.25] (0.center) to (1.center);
		\draw (6.center) to (1.center);
		\draw (3.center) to (4.center);
		\draw [bend right, looseness=1.50] (4.center) to (5.center);
		\draw [bend left, looseness=1.50] (6.center) to (7.center);
		\draw (8.center) to (7.center);
		\draw (9.center) to (10.center);
		\draw (5.center) to (11.center);
		\draw [bend left, looseness=1.50] (10.center) to (11.center);
		\draw [bend left, looseness=1.50] (8.center) to (9.center);
\end{scope}

\begin{scope}[rotate=90,xshift=-4.55cm, yshift=-7.35cm]
	
		\node  (0) at (-0.75, 0) {};
		\node  (1) at (-1.75, 0.5) {};
		\node  (2) at (1, 0) {};
		\node  (3) at (2, 0.5) {};
		\node  (4) at (3, 2) {};
		\node  (5) at (3, 3) {};
		\node  (6) at (-2.75, 2) {};
		\node  (7) at (-2.75, 3) {};
		\node  (8) at (-1.75, 4.5) {};
		\node  (9) at (-0.75, 5) {};
		\node (10) at (1, 5) {};
		\node  (11) at (2, 4.5) {};
		\node [style=nodee] (12) at (-1, 0.5) {};
		\node [style=nodee] (13) at (1.25, 0.5) {};
		\node [style=nodee] (14) at (2.75, 2.5) {};
		\node [style=nodee] (15) at (1.5, 4.5) {};
		\node [style=nodee] (16) at (-1.25, 4.5) {};
		\node [style=nodee] (17) at (-2.5, 2.5) {};

		\draw (0.center) to (2.center);
		\draw [bend right, looseness=1.50] (2.center) to (3.center);
		\draw [bend left, looseness=1.25] (0.center) to (1.center);
		\draw (6.center) to (1.center);
		\draw (3.center) to (4.center);
		\draw [bend right, looseness=1.50] (4.center) to (5.center);
		\draw [bend left, looseness=1.50] (6.center) to (7.center);
		\draw (8.center) to (7.center);
		\draw (9.center) to (10.center);
		\draw (5.center) to (11.center);
		\draw [bend left, looseness=1.50] (10.center) to (11.center);
		\draw [bend left, looseness=1.50] (8.center) to (9.center);
\end{scope}
\end{scope}

\begin{scope}[yscale=-1, yshift=-0.26cm]
	\begin{scope}[name=hedge, rotate=-32, xshift=0.1cm, yshift=-1.45cm]
		\def\vr{2pt}
		\node  (0) at (0, 0.7) {};
		\node  (1) at (0, 0) {};
		\node  (2) at (4, 0) {};
		\node (3) at (4, 0.7) {};
		\draw (0.center) to (3.center);
		\draw [in=180, out=0] (1.center) to (2.center);
		\draw [bend left=90, looseness=1.50] (3.center) to (2.center);
		\draw [bend right=90, looseness=1.50] (0.center) to (1.center);		
	\end{scope}

	\begin{scope}[name=hedge, rotate=-90, xshift=2.45cm, yshift=-2.85cm]
		\def\vr{2pt}
		\node  (0) at (0, 0.7) {};
		\node  (1) at (0, 0) {};
		\node  (2) at (4, 0) {};
		\node (3) at (4, 0.7) {};
		\draw (0.center) to (3.center);
		\draw [in=180, out=0] (1.center) to (2.center);
		\draw [bend left=90, looseness=1.50] (3.center) to (2.center);
		\draw [bend right=90, looseness=1.50] (0.center) to (1.center);		
	\end{scope}	
	
	\begin{scope}[name=hedge, rotate=32, xshift=-4.6cm, yshift=-6.75cm]
		\def\vr{2pt}
		\node  (0) at (0, 0.7) {};
		\node  (1) at (0, 0) {};
		\node  (2) at (4, 0) {};
		\node (3) at (4, 0.7) {};
		\draw (0.center) to (3.center);
		\draw [in=180, out=0] (1.center) to (2.center);
		\draw [bend left=90, looseness=1.50] (3.center) to (2.center);
		\draw [bend right=90, looseness=1.50] (0.center) to (1.center);		
	\end{scope}	
	\begin{scope}[rotate=90,xshift=-9.1cm]
		
		\node  (0) at (-0.75, 0) {};
		\node  (1) at (-1.75, 0.5) {};
		\node  (2) at (1, 0) {};
		\node  (3) at (2, 0.5) {};
		\node  (4) at (3, 2) {};
		\node  (5) at (3, 3) {};
		\node  (6) at (-2.75, 2) {};
		\node  (7) at (-2.75, 3) {};
		\node  (8) at (-1.75, 4.5) {};
		\node  (9) at (-0.75, 5) {};
		\node (10) at (1, 5) {};
		\node  (11) at (2, 4.5) {};
		\node [style=nodee] (12) at (-1, 0.5) {};
		\node [style=nodee] (13) at (1.25, 0.5) {};
		\node [style=nodee] (14) at (2.75, 2.5) {};
		\node [style=nodee] (15) at (1.5, 4.5) {};
		\node [style=nodee] (16) at (-1.25, 4.5) {};
		\node [style=nodee] (17) at (-2.5, 2.5) {};

		\draw (0.center) to (2.center);
		\draw [bend right, looseness=1.50] (2.center) to (3.center);
		\draw [bend left, looseness=1.25] (0.center) to (1.center);
		\draw (6.center) to (1.center);
		\draw (3.center) to (4.center);
		\draw [bend right, looseness=1.50] (4.center) to (5.center);
		\draw [bend left, looseness=1.50] (6.center) to (7.center);
		\draw (8.center) to (7.center);
		\draw (9.center) to (10.center);
		\draw (5.center) to (11.center);
		\draw [bend left, looseness=1.50] (10.center) to (11.center);
		\draw [bend left, looseness=1.50] (8.center) to (9.center);
\end{scope}

\begin{scope}[rotate=90,xshift=-4.55cm, yshift=-7.35cm]
	
		\node  (0) at (-0.75, 0) {};
		\node  (1) at (-1.75, 0.5) {};
		\node  (2) at (1, 0) {};
		\node  (3) at (2, 0.5) {};
		\node  (4) at (3, 2) {};
		\node  (5) at (3, 3) {};
		\node  (6) at (-2.75, 2) {};
		\node  (7) at (-2.75, 3) {};
		\node  (8) at (-1.75, 4.5) {};
		\node  (9) at (-0.75, 5) {};
		\node (10) at (1, 5) {};
		\node  (11) at (2, 4.5) {};
		\node [style=nodee] (12) at (-1, 0.5) {};
		\node [style=nodee] (13) at (1.25, 0.5) {};
		\node [style=nodee] (14) at (2.75, 2.5) {};
		\node [style=nodee] (15) at (1.5, 4.5) {};
		\node [style=nodee] (16) at (-1.25, 4.5) {};
		\node [style=nodee] (17) at (-2.5, 2.5) {};

		\draw (0.center) to (2.center);
		\draw [bend right, looseness=1.50] (2.center) to (3.center);
		\draw [bend left, looseness=1.25] (0.center) to (1.center);
		\draw (6.center) to (1.center);
		\draw (3.center) to (4.center);
		\draw [bend right, looseness=1.50] (4.center) to (5.center);
		\draw [bend left, looseness=1.50] (6.center) to (7.center);
		\draw (8.center) to (7.center);
		\draw (9.center) to (10.center);
		\draw (5.center) to (11.center);
		\draw [bend left, looseness=1.50] (10.center) to (11.center);
		\draw [bend left, looseness=1.50] (8.center) to (9.center);
\end{scope}
\end{scope}

\begin{scope}
\begin{scope}[name=hedge, rotate=-32, xshift=-8.6cm, yshift=-1.45cm]
		\def\vr{2pt}
		\node  (0) at (0, 0.7) {};
		\node  (1) at (0, 0) {};
		\node  (2) at (4, 0) {};
		\node (3) at (4, 0.7) {};
		\draw (0.center) to (3.center);
		\draw [in=180, out=0] (1.center) to (2.center);
		\draw [bend left=90, looseness=1.50] (3.center) to (2.center);
		\draw [bend right=90, looseness=1.50] (0.center) to (1.center);		
	\end{scope}

	\begin{scope}[name=hedge, rotate=-150, xshift=4.45cm, yshift=-1.5cm]
		\def\vr{2pt}
		\node  (0) at (0, 0.7) {};
		\node  (1) at (0, 0) {};
		\node  (2) at (4, 0) {};
		\node (3) at (4, 0.7) {};
	\end{scope}	
	
	\begin{scope}[name=hedge, rotate=90, xshift=-1.8cm, yshift=9.55cm]
		\def\vr{2pt}
		\node  (0) at (0, 0.7) {};
		\node  (1) at (0, 0) {};
		\node  (2) at (4, 0) {};
		\node (3) at (4, 0.7) {};
		\draw (0.center) to (3.center);
		\draw [in=180, out=0] (1.center) to (2.center);
		\draw [bend left=90, looseness=1.50] (3.center) to (2.center);
		\draw [bend right=90, looseness=1.50] (0.center) to (1.center);		
	\end{scope}	
	
	\begin{scope}[rotate=90,xshift=-4.45cm, yshift=7.4cm]
			\node  (0) at (-0.75, 0) {};
		\node  (1) at (-1.75, 0.5) {};
		\node  (2) at (1, 0) {};
		\node  (3) at (2, 0.5) {};
		\node  (4) at (3, 2) {};
		\node  (5) at (3, 3) {};
		\node  (6) at (-2.75, 2) {};
		\node  (7) at (-2.75, 3) {};
		\node  (8) at (-1.75, 4.5) {};
		\node  (9) at (-0.75, 5) {};
		\node (10) at (1, 5) {};
		\node  (11) at (2, 4.5) {};
		\node [style=nodee] (12) at (-1, 0.5) {};
		\node [style=nodee] (13) at (1.25, 0.5) {};
		\node [style=nodee] (14) at (2.75, 2.5) {};
		\node [style=nodee] (15) at (1.5, 4.5) {};
		\node [style=nodee] (16) at (-1.25, 4.5) {};
		\node [style=nodee] (17) at (-2.5, 2.5) {};

		\draw (0.center) to (2.center);
		\draw [bend right, looseness=1.50] (2.center) to (3.center);
		\draw [bend left, looseness=1.25] (0.center) to (1.center);
		\draw (6.center) to (1.center);
		\draw (3.center) to (4.center);
		\draw [bend right, looseness=1.50] (4.center) to (5.center);
		\draw [bend left, looseness=1.50] (6.center) to (7.center);
		\draw (8.center) to (7.center);
		\draw (9.center) to (10.center);
		\draw (5.center) to (11.center);
		\draw [bend left, looseness=1.50] (10.center) to (11.center);
		\draw [bend left, looseness=1.50] (8.center) to (9.center);
\end{scope}

	\begin{scope}[rotate=90,xshift=4.6cm, yshift=7.4cm]
			\node  (0) at (-0.75, 0) {};
		\node  (1) at (-1.75, 0.5) {};
		\node  (2) at (1, 0) {};
		\node  (3) at (2, 0.5) {};
		\node  (4) at (3, 2) {};
		\node  (5) at (3, 3) {};
		\node  (6) at (-2.75, 2) {};
		\node  (7) at (-2.75, 3) {};
		\node  (8) at (-1.75, 4.5) {};
		\node  (9) at (-0.75, 5) {};
		\node (10) at (1, 5) {};
		\node  (11) at (2, 4.5) {};
		\node [style=nodee] (12) at (-1, 0.5) {};
		\node [style=nodee] (13) at (1.25, 0.5) {};
		\node [style=nodee] (14) at (2.75, 2.5) {};
		\node [style=nodee] (15) at (1.5, 4.5) {};
		\node [style=nodee] (16) at (-1.25, 4.5) {};
		\node [style=nodee] (17) at (-2.5, 2.5) {};

\end{scope}

\end{scope}

	\end{tikzpicture}

(c) Cut type II
	\endminipage
	
	\caption{Hypergraph $H$ and its convex cuts.}
	\label{fig:t}
\end{figure}

 There are two different types of cuts in $H$. 	
The cut of type I consists of the central $6$-edge and three $2$-edges that do not intersect it as can be seen in Figure~\ref{fig:t}(b). A cut of type II consists of a non-central $6$-edge and its opposite $2$-edge as can be seen in Figure~\ref{fig:t}(c). Both cuts are convex and also the conclusion of Proposition~\ref{prop:2} holds. This, together with the fact that $E(H)$ partitions into one cut of type I and six cuts of type II, allows us to use the cut method to calculate Wiener index of $H$ as
\begin{align*}
	W(H) = \binom{6}{2}7\cdot 7 + 6\left(\binom{4}{2} +4\cdot7 + 4\cdot 31 + 7\cdot 31 \right) =2985.		
\end{align*}

\section*{Acknowledgements}
 
This work has been supported by the financial support from the Slovenian Research Agency (research core funding P1-0297 and projects J1-2452 and N1-0285).

\section*{Declaration of interests}
 
The authors declare that they have no conflict of interest. 
\section*{Data availability}
 
Our manuscript has no associated data.

\end{document}